\documentclass{article}

\usepackage{color}
 \catcode`@=11
\usepackage{enumitem}   
\usepackage{color}
\usepackage{amsthm,amssymb,amsfonts,mathrsfs}
\usepackage{amsmath}
\usepackage[toc,page]{appendix}
\catcode`@=11 \@addtoreset{equation}{section}

\catcode`@=12

\newtheorem{theorem}{Theorem}[section]
\newtheorem{proposition}{Proposition}[section]
\newtheorem{lemma}{Lemma}[section]

\theoremstyle{definition}

\newtheorem{definition}{Definition}[section]
\newtheorem{remark}{Remark}

\usepackage[colorlinks=true,urlcolor=red,linkcolor=blue,citecolor=red,bookmarks=red]{hyperref}
\title{Controllability for a $2\times 2$ nonlinear degenerate parabolic system via one boundary control force}
\author{
{\sc Margarita Arias$^{a}$, Abdelkarim Hajjaj$^{b}$, Amine Sbai$^{a,b}$}\\\\
$^a$Department of Applied Mathematics,\\ University of Granada, Granada, Spain.\\\\
$^b$Hassan First University of Settat\\
Faculty of Sciences and Technology, MISI Laboratory,\\
B.P. 577, Settat 26000, Morocco.
}

\date{}
\begin{document}

\maketitle

\begin{abstract}
In this paper we study the local boundary controllability for a non linear system of two degenerate parabolic equations with a control acting on only one equation. We analyze boundary null controllability properties for the linear system via the moment method by Fattorini and Russell, together with some results on biorthogonal families. Moreover, we provide an estimate on the null-control cost. This estimate let us prove a local exact boundary controllability result to zero of the nonlinear system following the iterative method from Lebeau and Robbiano as in \cite{Burgos_2020, Liu_2012}. 
\end{abstract}


\section{Introduction}\label{section 1}
The problem of the null controllability for nondegenerate parabolic equations is already well understood, and there have been many works on this issue in recent years \cite{barbu, burgosgarcia2006, BurTer2010, fursimanuv1996, guerrero2007}. This is not the case when we consider diffusions with some kind of degeneracy, which are so common in many applications.

This work deals with the boundary controllability properties of the nonlinear degenerate system 
\begin{equation}\label{bound-nonlinear-sys}
\left\{
\begin{array}{lll}
y_t - ( x^{\alpha}  y_x )_x - A y = f(y) ,  &  & \text{in} \;  Q_T:=(0, \, T)\times(0, 1), \\
y(t,1)= B v(t), & & \text{in} \; (0, \, T),  \\
\begin{cases}
\; y(t, 0) = 0,   \qquad   \, 0 \leq \alpha < 1 \\
\;  x^{\alpha} y_x  (t, 0)= 0 , \quad 1\leq \alpha < 2 \\
\end{cases}  & & t \in (0, T),\\
y(0,x)= y_{0}(x),  & & \text{in}  \;  (0, 1).
\end{array}
\right.
\end{equation}
Here, \( y = (y_1, y_2)^* \) is the state variable, while \( v \in L^2(0, T) \) represents the control function, which acts at a single boundary point for all times. The parameter \( \alpha \in [0, 2) \) indicates the order of degeneracy of the diffusion coefficient, which may vanish at \( x = 0 \). The initial condition, \( y_0 \), and the reaction term, \( f(y) \), are given functions in appropriate function spaces, which will be defined in the forthcoming sections.
Furthermore, \( A \in \mathcal{L}(\mathbb{R}^2) \) and \( B \in \mathbb{R}^2 \) are selected to be a suitable coupling matrix and control operator, respectively, such that the condition
\begin{equation}\label{rank}
rank [B | AB] = 2.
\end{equation}
is satisfied.

We are interested in the null-controllability of such systems. Recall that system \eqref{bound-nonlinear-sys} is said to be null-controllable if, for any initial condition \( y_0 \), there exists a control function \( v \) such that the corresponding solution \( y \) of \eqref{bound-nonlinear-sys} (as defined in Theorem \ref{thm-local-exact-control} at the end of this section) satisfies
\[
y(T, x) = 0 \quad \text{for all } x \in (0, 1).
\]
The first study on the distributed null-controllability of degenerate coupled parabolic systems was presented in \cite{CanTer2009}. In particular, the authors consider a cascade system with the same diffusion coefficient
\[
a(x) = x^\alpha \quad \text{where } \alpha \in (0, 2),
\]
and recover distributed controllability results similar to those obtained in \cite{BurTer2010}. For more general systems of degenerate equations, we refer the reader to \cite{Hajjaj2013, Hajjaj2011}.

Furthermore, the local controllability of nonlinear systems controlled by a single distributed control force has been investigated in \cite{Khodja2003}, where the focus was on a \( 2 \times 2 \) nonlinear parabolic system.

In previous studies, controllability results for both linear and nonlinear problems were based on global Carleman estimates for scalar nondegenerate parabolic problems. Our approach, however, is distinct in that it employs spectral techniques. To the best of our knowledge, this is the first investigation into the boundary controllability properties of a nonlinear degenerate coupled system using this methodology.

The main objective of this work is to address the local controllability problem for the degenerate system \eqref{bound-nonlinear-sys}, where the control is applied at the boundary point \( x = 1 \), away from the degeneracy.

In order to prove the controllability at time \( T > 0 \) of system \eqref{bound-nonlinear-sys}, we will rewrite the controllability problem as a fixed-point problem for a suitable operator in appropriate function spaces. To implement this fixed-point strategy, we will first investigate the controllability properties of the following associated linear system
\begin{equation}\label{bound-sys}
\left\{
\begin{array}{lll}
y_t - ( x^{\alpha}  y_x )_x = A y ,  &  & \text{in} \;  Q_T, \\
y(t,1)= B v(t), & & \text{in} \; (0, \, T),  \\
\begin{cases}
 \; y(t, 0) = 0,   \qquad   \, 0 \leq \alpha < 1 \\
\;  x^{\alpha} y_x  (t, 0)= 0 , \quad 1\leq \alpha < 2 \\
\end{cases}  & & t \in (0, T),\\
y(0,x)= y_{0}(x),  & & \text{in}  \;  (0, 1).
\end{array}
\right.
\end{equation}
In \cite{FBT}, the authors assert that a necessary condition for the controllability of this type of system is given by the so-called Kalman's rank condition \eqref{rank}. As explained in \cite{FBT}, by taking \( P = [B \mid A B] \), the change of variables
$$
\tilde{y}=P^{-1} y,
$$
leads to the following reformulation of \eqref{bound-sys}:
$$
\begin{cases}\tilde{y}_t-(x^\alpha\tilde{y}_{x})_x=\tilde{A} \tilde{y}, & (t, x) \in(0, T) \times(0,1), \\ \tilde{y}(t, 1)=\tilde{B} v, & t \in(0, T), \\
\begin{cases}
\; \tilde{y}(t, 0) = 0,   \qquad   \, 0 \leq \alpha < 1 \\
\;  x^{\alpha} \tilde{y}_x  (t, 0)= 0 , \quad 1\leq \alpha < 2 \\
\end{cases}  &  t \in (0, T),\\
\tilde{y}(0, x)=P^{-1} y_0(x), & x \in(0,1),\end{cases}
$$
with
$$
\tilde{A}=\left(\begin{array}{ll}
0 & a_1 \\
1 & a_2
\end{array}\right) \text { and } \tilde{B}=e_1=\binom{1}{0} .
$$

Therefore, passing through this mentioned change of variables, the situation reduces to the case where
\begin{equation}\label{Vec-sys-matrix}
A=\left(\begin{array}{ll}
0 & a_1 \\
1 & a_2
\end{array}\right) \quad \text { and } \quad B=e_1=\binom{1}{0} .
\end{equation}
For simplicity, it will be assumed in the rest of the paper that $A$ and $B$ are given by \eqref{Vec-sys-matrix}.

We are particularly interested in the study of system \eqref{bound-nonlinear-sys} under the assumption that the coupling matrix \( A \) admits two distinct eigenvalues, i.e., 
\begin{align}\label{condition on A}
a_2^2 + 4 a_1 \neq 0.
\end{align}
We establish the boundary null-controllability at time \( T > 0 \) of system \eqref{bound-sys} in the framework of the space \( H_{\alpha}^{-1}(0,1)^2 \), which will be defined in Section \ref{Section-prel}.
Indeed, denoting, as usual, by \( \{ j_{\nu, n} \}_{n \geq 1} \) the sequence of positive zeros of the Bessel function of the first kind and order \( \nu \in \mathbb{R} \), \( J_{\nu} \), and taking \( \kappa_{\alpha} = \frac{2 - \alpha}{2} > 0 \), we will prove:
\begin{theorem}\label{Thm-null-cont1}
Let $\alpha \in [0,2)$ and denote $\mu_1$ and $\mu_2$ the eigenvalues of $A$. Then system \eqref{bound-sys} is null controllable at any time $T >0$
if and only if \eqref{rank} holds and also
\begin{equation}\label{spectre-cond}
\kappa_{\alpha}^2 ( j_{\nu_{\alpha},n}^2 - j_{\nu_{\alpha},l}^2 ) \neq \mu_2 - \mu_1,\quad \forall n, l \in \mathbb{N}^*, \quad \text{with} \quad n\neq l.
\end{equation}
Moreover, there exist two positive constants $C_0$ and $M$ independent of $T$, such that for any $T>0$ there is a bounded linear operator $\mathcal{C}_{T}^{(0)}: H_{\alpha}^{-1}(0,1)^2 \rightarrow L^{2}(0, T)$ satisfying
\begin{equation}\label{OperatorC0bound}
\left\|\mathcal{C}_{T}^{(0)}\right\|_{\mathcal{L}\left(H_{\alpha}^{-1}(0,1)^2, L^{2}(0, T)\right)} \leq C_{0} e^{M / T}
\end{equation}
and such that the solution
$$
y=(y_1, y_2) \in L^{2}\left(Q_{T} ; \mathbb{R}^{2}\right) \cap C^{0}\left([0, T] ; H_{\alpha}^{-1}(0,1)^2\right)
$$
of system \eqref{bound-sys} associated to $y_{0}=\left(y^{0}_{1}, y^0_{2}\right) \in H_{\alpha}^{-1}(0,1)^2$ and $v=\mathcal{C}_{T}^{(0)}\left(y_{0}\right)$ satisfies $y(\cdot, T)=$
$0 .$
\end{theorem}

Before showing the controllability result for the nonlinear system \eqref{bound-nonlinear-sys}, we will first adapt an iterative method for the null-controllability in presence of source terms (see for instance \cite{Burgos_2020, Liu_2012}) to the following system
\begin{equation}\label{bound-nonhomog-sys}
\left\{
\begin{array}{lll}
y_{t} - ( x^{\alpha}  y_x )_x - A y = f ,  &  & \text{in} \;  Q_T, \\
y(t,1)= B v(t), & & \text{in} \; (0, \, T),  \\
\begin{cases}
\; y(t, 0) = 0,   \qquad   \, 0 \leq \alpha < 1 \\
\;   x^{\alpha} y_x  (t, 0)= 0 , \quad 1\leq \alpha < 2 \\
\end{cases}  & & t \in (0, T),\\
y(0,x)= y_{0}(x),  & & \text{in}  \;  (0, 1),
\end{array}
\right.
\end{equation}
where $y_{0} \in H_{\alpha}^{-1}(0,1)^2$ and $f$ is a given function satisfying appropriate assumptions. Specifically,  denoting by $\mathcal{F}, \mathcal{V}, \mathcal{H}_{0}$ and $\mathcal{H}$ the weighted spaces  defined in section \ref{Section-null}, one has:
\begin{theorem}\label{thm-operators}
 Suppose that \eqref{rank} and \eqref{spectre-cond} hold. Then, for every $T>0$, there exist two bounded linear operators
$$\begin{aligned} \mathcal{C}_{T}^{(1)}: H_{\alpha}^{-1}(0,1)^2 \times \mathcal{F}  \rightarrow \mathcal{V}& \quad \text { and } \quad E_{T}^{(0)}: & H_{\alpha}^{-1}(0,1)^2 \times  \mathcal{F}  \rightarrow   \mathcal{H}_{0}\\
(y_0,f)\rightarrow v&   &(y_0,f)\rightarrow y
\end{aligned}$$
and positive constants $C$ independent of $T$, possibly different, such that
\begin{enumerate}
\item[a)] $\left\|\mathcal{C}_{T}^{(1)}\right\|_{\mathcal{L}\left(H_{\alpha}^{-1}(0,1)^2 \times \mathcal{F}, \mathcal{V}\right)} \leq C e^{C\left(T+\frac{1}{T}\right)}$ and $\left\|E_{T}^{(0)}\right\|_{\mathcal{L}\left(H_{\alpha}^{-1}(0,1)^2 \times \mathcal{F}, \mathcal{H}_{0}\right)} \leq C e^{C\left(T+\frac{1}{T}\right)}$. 
\item[b)] The restriction $ E_{T}^{(1)}:=\left.E_{T}^{(0)}\right|_{H^{-1}_\alpha(0, 1) \times H_{\alpha}^{1}(0, 1) \times \mathcal{F}} \in \mathcal{L}\left(H^{-1}_\alpha(0, 1) \times H_{\alpha}^{1}(0, 1) \times \mathcal{F}, \mathcal{H} \right)$, \\and $\left\|E_{T}^{(1)}\right\|_{\mathcal{L}\left(H^{-1}_\alpha(0, 1) \times H_{\alpha}^{1}(0, 1) \times \mathcal{F}, \mathcal{H}\right)} \leq C e^{C\left(T+\frac{1}{T}\right)}$.
\item[c)] Fixed $\left(y_{0}, f\right) \in H_{\alpha}^{-1}(0,1)^2 \times \mathcal{F}$ (resp., $\left(y_{0}, f\right) \in H^{-1}_\alpha(0, 1) \times H_{\alpha}^{1}(0, 1) \times \mathcal{F}$), the solution of \eqref{bound-nonhomog-sys} associated to $\left(y_{0}, f\right)$ with  $v=\mathcal{C}_{T}^{(1)}\left(y_{0}, f\right)$
 is $$y=E_{T}^{(0)}\left(y_{0}, f\right) \in \mathcal{H}_{0}  \mbox{ (resp. }\; y=E_{T}^{(1)}\left(y_{0}, f\right) \in \mathcal{H}).$$
  \end{enumerate}
\end{theorem}
As mentioned earlier, the proof of this result is based on an iterative method. We employ an idea that originates from Lebeau and Robbiano \cite{LR1995}, which consists of controlling a part of the state trajectory and utilizing the dissipative nature of the parabolic equations to drive the remaining part to zero. See also \cite{TenTucs}. Finally, we apply a fixed-point method introduced in \cite{Imanuvilov}.

The previous results enable us to obtain a local exact controllability result for the nonlinear system \eqref{bound-nonlinear-sys}. Specifically, we will prove:
\begin{theorem}\label{thm-local-exact-control}
Suppose conditions \eqref{rank} and \eqref{spectre-cond} hold and that the source term $f$ satisfies 
\begin{align}\label{nonlinCondit1}
\begin{cases}
\mid f_i(y_1,y_2) - f_i(z_1,z_2) \mid \leq C \| y_2^2 - z_2^2 \|,\, i=1,2 \\[.1cm]
f_i(y_1,0)=0,\, i=1,2
\end{cases}
\end{align} 
Then, fixed $T>0$, there exists $\delta>0$ such that if $\left(y^{0}_{1}, y^0_{2}\right) \in$ $H^{-1}_\alpha(0, 1) \times H_{\alpha,0}^{1}(0, 1)$ fulfills
\begin{equation}\label{estim-1.15}
\left\|y^{0}_{1}\right\|_{H^{-1}_\alpha}+\left\|y^0_{2}\right\|_{H_{\alpha,0}^{1}} \leq \delta,
\end{equation}
there exists $v \in L^{2}(0, T)$ for which  \eqref{bound-nonlinear-sys} has a unique solution
\begin{align*}
y=(y_1, y_2) \in\left[L^{2}\left(Q_{T}\right) \cap C^{0}\left([0, T] ; H^{-1}_{\alpha}(0,1)^2\right)\right] \times C^{0}\left(\overline{Q}_{T}\right)
\end{align*}
which satisfies $y(T,.)=0$. 
\end{theorem}

\subsection*{Plan of the paper}
The rest of the paper is organized as follows. In Section \ref{Section-prel}, we prove the well-posedness of system \eqref{bound-nonhomog-sys}, and consequently that of system \eqref{bound-sys}, in appropriate weighted spaces using the transposition method. In Section \ref{Section-eigenvalue}, we discuss the spectral analysis related to scalar degenerate operators and provide a description of the spectrum associated with system \eqref{bound-sys}, which will be useful for developing the moment method. Section \ref{Section-approx} is devoted to studying the boundary null controllability problem for the linear system \eqref{bound-sys} and establishing an estimate of the control cost. Using these results, we will prove in Section \ref{Section-null} the controllability result for the nonlinear system \eqref{bound-nonlinear-sys} via a fixed-point strategy. The paper concludes with the Appendix, where we provide some proofs to make the paper self-contained.

\section{Preliminary results}\label{Section-prel}
\subsection{Functional framework}
Let us start introducing the functional setting associated with degenerate operators. We distinguish here between two cases: \(0 \leq \alpha < 1\) and \(1 \leq \alpha < 2\), primarily for reasons of well-posedness. In the natural functional setting for weakly degenerate operators, specifically when \(0 \leq \alpha < 1\), the trace at \(x = 0\) exists which allows to consider a Dirichlet condition at $x=0$. On the other hand, the trace does not exist when $\alpha \geq 1$. Here the Dirichlet boundary condition needs to be changed by some Neumann-kind one. This leads us to consider, as in \cite{Alabau2006, CMV2008, CanTer2009}, the following weighted Hilbert spaces:

\paragraph*{} For the Weakly degenerate case (WD) $0\leq \alpha<1$, we define:
\begin{align*}
\bullet &H_{\alpha}^1(0,1):= \Big\{ u \in L^2(0,1): u\, \text{absolutely continuous in}\, [0,1], x^{\alpha/2}u_x \in L^2(0,1)\,\text{and}\, u(1)=u(0)=0 \Big\}\\[.1cm]
\bullet  &H_{\alpha}^2(0,1):= \Big\{ u \in H_{\alpha}^1(0, 1):  x^{\alpha} u_x \in H^1(0,1)\Big\}.
\end{align*}

\paragraph*{} For the Strongly degenerate case (SD) $1\leq \alpha<2$, we define
\begin{align*}
\bullet  H_{\alpha}^1(0,1):=& \Big\{ u \in L^2(0,1): u\, \text{locally absolutely continuous in}\, (0,1], x^{\alpha/2}u_x \in L^2(0,1)\,\text{and}\, u(1)=0 \Big\}\\[.1cm]
\bullet  H_{\alpha}^2(0,1):=& \Big\{ u \in H_{\alpha}^1(0, 1):  x^{\alpha} u_x \in H^1(0,1)\Big\}\\
=&\Big\{ u \in L^2(0,1): u\, \text{locally absolutely continuous in}\, (0,1], x^{\alpha} u\in H^1_0(0,1), x^{\alpha} u_x \in H^1(0,1)\\
\quad&   \,\text{and}\, ( x^{\alpha} u_x)(0)=0 \Big\}.
\end{align*}
In both cases, the norms are defined as follows
\begin{align*}
\| u \|_{H_{\alpha}^1}^2:=  \| u \|_{L^2(0,1)}^2 +  \|  x^{\alpha/2} u_x \|_{L^2(0,1)}^2,\qquad
\| u \|_{H_{\alpha}^2}^2:=  \| u \|_{H_{\alpha}^1}^2 +  \| ( x^{\alpha} u_x )_x \|_{L^2(0,1)}^2.
\end{align*}
Let $H_{\alpha}^{-1}(0,1)$ be the dual space of $H_{\alpha}^1(0,1)$ with respect to the pivot space
$L^2(0,1)$, endowed with the natural norm
\begin{equation*}
\| z \|_{H_{\alpha}^{-1}}:=  \sup_{\| u \|_{H_{\alpha}^{1}}=1} \langle z, u \rangle_{H_{\alpha}^{-1}, H_{\alpha}^{1}}.
\end{equation*}
In what follows, for simplicity, we will always denote
by $\langle \cdot,\cdot \rangle$ the standard scalar product of either $L^2(0,1)$ or $L^2(0,1)^2$, by $\langle \cdot,\cdot \rangle_{X',X}$ the duality
pairing between the Hilbert space $X$ and its dual $X'$.  On the other hand, we will use $\|\cdot \|_{H_{\alpha}^1}$ (resp. $\|\cdot \|_{H_{\alpha}^{-1}}$) for denoting the norm of $H_{\alpha}^1(0,1)^2$ (resp. $H_{\alpha}^{-1}(0,1)^2)$.

\subsection{Well-posedness}\label{section: Well Posedness}
In this section, we present results related to the existence, uniqueness, and continuous dependence on the data for the linear problem \eqref{bound-nonhomog-sys}.

Before going any further, note that system \eqref{bound-sys} is a linearization of system \eqref{bound-nonhomog-sys} around the equilibrium point \( (0, 0) \). In particular, \eqref{bound-sys} can also be written as \eqref{bound-nonhomog-sys} with \( f = (0, 0) \). To this end, we will first address the well-posedness of system \eqref{bound-nonhomog-sys}, and from this, the well-posedness of system \eqref{bound-sys} will follow. For this purpose, let us consider the linear backward-in-time problem:
\begin{equation}\label{adjoint system}
\left\{\begin{array}{lll}
-\varphi_{t}-( x^{\alpha} \varphi_x )_x =A^{*}\varphi + g & & \text { in } Q_{T} \\
\varphi(t, 1)=0 & & t \in (0, T)\\
\left\{\begin{array}{l}
\varphi(t, 0)=0, \quad 0 \leq \alpha<1 \\
x^{\alpha} \varphi_{x}(t, 0)=0, \quad 1 \leq \alpha<2
\end{array}\right. & & \text { on }(0, T) \\
\varphi(T,\cdot)=\varphi_{0} & & \text { in }(0, 1)
\end{array}\right.
\end{equation}
where $A$ is given in \eqref{Vec-sys-matrix} and $\varphi_{0}$ and $g$ are functions in appropriate spaces.

Let us start with a first result on the existence and uniqueness of strong solutions to system \eqref{adjoint system}. 
which is by now classical (see, for instance \cite[Theorem 2.1]{CMV2008}).  One has:
\begin{proposition}\label{prop.2.1} Let us assume that $\varphi_{0} \in H_{\alpha}^{1}\left(0, 1 \right)^{2}$  and $g \in L^{2}\left(Q_{T} ; \mathbb{R}^{2}\right) $. Then, system \eqref{adjoint system} has a unique strong solution 
$$
\varphi \in L^{2}\left(0, T ; H^{2}_\alpha\left(0, 1\right)^{2} \right) \cap  C^{0}\left([0, T] ; H^{1}_{\alpha}\left(0, 1\right)^2\right) 
$$
In addition, there exists a positive constant $C$ such that
\begin{equation}\label{bound phi prop 2.1}
\|\varphi\|_{L^2\left(0, T ; H_\alpha^2(0,1)^2\right)}+\|\varphi\|_{C^0\left([0, T] ; H_\alpha^1(0,1)^2\right)} \leq C\left(\left\|\varphi_0\right\|_{H_\alpha^1}+\|g\|_{L^2(L^2)}\right)
\end{equation}
\end{proposition}

In view of Proposition \ref{prop.2.1}, we can define solution by transposition to system  \eqref{bound-nonhomog-sys}. Taking into account that the notation $\displaystyle y \cdot g$ represents the scalar product in $\mathbb{R}^2$.
\begin{definition}
Let $y_{0} \in H^{-1}_{\alpha}\left(0, 1\right)^{2}, v \in L^{2}(0, T)$ and $f \in L^{2}\left(Q_{T} ; \mathbb{R}^{2}\right)$ be given. It will be said that $y \in L^{2}\left(Q_{T} ; \mathbb{R}^{2}\right)$ is a solution by transposition to \eqref{bound-nonhomog-sys} if, for each $g \in L^{2}\left(Q_{T} ; \mathbb{R}^{2}\right)$, one has
\begin{equation}\label{duality identity}
\iint_{Q_{T}} y \cdot g\,  d x d t=\left\langle y_{0}, \varphi(\cdot, 0)\right\rangle-\int_{0}^{T} B^{*} \left(x^\alpha \varphi_{x}\right)(1, t) v(t) d t+\iint_{Q_{T}} f \cdot \varphi d x d t
\end{equation}
 where $\varphi \in L^{2}\left(0, T ; H^{2}_\alpha\left(0, 1\right)^{2} \right) \cap  C^{0}\left([0, T] ; H^{1}_{\alpha}\left(0, 1\right)^2\right)$ is the solution of \eqref{adjoint system} associated to $g$ and $\varphi_{0}=(0,0)$ (recall that $\langle\cdot, \cdot\rangle$ stands for the usual duality pairing between $H_{\alpha}^{-1}(0,1)^2$ and $\left.H_{\alpha}^1(0,1)^2\right)$.
\end{definition}
With this definition we have:
\begin{proposition}\label{well.posed.transp}
Let us assume that $y_{0}= \left(y^{0}_{1}, y^0_{2}\right) \in H^{-1}_{\alpha}(0,1)^2, v \in L^{2}(0, T)$ and $f \in$ $L^{2}\left(Q_{T} ; \mathbb{R}^{2}\right)$. Then, system \eqref{bound-nonhomog-sys} admits a unique solution by transposition $y=\left(y_{1}, y_{2}\right)$ that satisfies
$$
\left\{\begin{array}{l}
y \in L^{2}\left(Q_{T} ; \mathbb{R}^{2}\right) \cap C^{0}\left([0, T] ; H_{\alpha}^{-1}(0,1)^2\right), \quad y_{t} \in L^{2}\left(0, T ; \left(H^{2}_{\alpha}(0,1)^2 \right)^{\prime}\right), \\
y_{t}- \left(x^\alpha y_{x}\right)_x-A y=f \text { in } L^{2}\left(0, T ;\left(H^{2}_{\alpha}(0,1)^2 \right)^{\prime}\right), \\
y(\cdot, 0)=y_{0} \text { in } H^{-1}_{\alpha}\left(0, 1 \right)^{2}
\end{array}\right.
$$
and
\begin{equation}\label{bounds.prop.2.2}
\|y\|_{L^{2}\left(L^{2}\right)}+\|y\|_{C^{0}\left(H^{-1}_{\alpha}\right)}+\left\|y_{t}\right\|_{L^{2}\left(\left(H^{2}_{\alpha}(0,1)^2 \right)^{\prime}\right)} \leq C e^{C T}\left(\left\|y_{0}\right\|_{H^{-1}_{\alpha}}+\|v\|_{L^{2}(0, T)}+\|f\|_{L^{2}\left(L^{2}\right)}\right),
\end{equation}
 for a constant $C = C(T) > 0$. Moreover
\begin{enumerate}
\item If $y_{2}^{0} \in L^{2}(0, 1)$, then $y_{2} \in L^{2}\left(0, T ; H_{\alpha}^{1}(0, 1)\right) \cap C^{0}\left([0, T] ; L^{2}(0, 1)\right)$ and, for a new constant $C>0$, one has
\begin{equation}\label{bounds.prop.2.2 a}
\|y_{2}\|_{L^{2}\left(H_{\alpha}^{1}\right)}+\|y_{2}\|_{C^{0}\left(L^{2}\right)} \leq C\left(\|y\|_{L^{2}\left(L^{2}\right)}+\left\|y^0_{2}\right\|_{L^{2}}+\|f\|_{L^{2}\left(L^{2}\right)}\right) .
\end{equation}
\item If $y^0_{2} \in H_{\alpha}^{1}(0, 1)$, then $y_{2} \in L^{2}\left(0, T ; H^{2}_{\alpha}(0, 1) \right) \cap C^{0}\left([0, T] ; H_{\alpha}^{1}(0, 1)\right)$ and, for a new constant $C>0$, one has
\begin{equation}\label{bounds.prop.2.2 b}
\|y_{2}\|_{L^{2}\left(H^{2}_{\alpha} \right)}+\|y_{2}\|_{C^{0}\left(H_{\alpha}^{1}\right)} \leq C\left(\|y\|_{L^{2}\left(L^{2}\right)}+\left\|y^0_{2}\right\|_{H_{\alpha}^{1}}+\|f\|_{L^{2}\left(L^{2}\right)}\right),
\end{equation}
and, in particular, $y=(y_1, y_2) \in L^{2}\left(Q_{T}\right) \times C^{0}\left(\overline{Q}_{T}\right)$.
 \end{enumerate}
\end{proposition}

The results shown in proposition \ref{well.posed.transp} are related to a similar ones proved in \cite{AHSS2021} for degenerate problems with singular potentials, based on some ideas from \cite{FBT}. We will include their proof in the Appendix for the reader’s convenience.
\section{Spectral properties of the operators}\label{Section-eigenvalue}

\subsection{Spectral properties of scalar degenerate operators}\label{subsec_spect_BC}
The knowledge of the eigenvalues and associated eigenfunctions of the degenerate diffusion operator \( y \mapsto -(x^{\alpha}y')' \) will be essential for our purposes. It is worth noting that explicit expressions for the eigenvalues are provided in \cite{Gueye} for the weakly degenerate case \( \alpha \in (0, 1) \), and in \cite{moyano} for the strongly degenerate case \( \alpha \in [1, 2) \). These eigenvalues depend on the Bessel functions of the first kind \cite{Watson}.

For this reason, we begin by presenting a brief overview of some results concerning the Bessel functions, which will be useful in the remainder of this paper.

For a real number $\nu$, we denote by $J_{\nu}$ the Bessel function of the first kind of order $\nu$ defined by :
$$
J_{\nu} (x) = \sum_{m\geq 0}  \frac{(-1)^{m}}{m! \;\Gamma(1+\nu+m)} \Big(\frac{x}{2}\Big)^{2m+\nu},
$$
where $\Gamma(.)$ is the Gamma function.

We recall that $\forall \nu \in \mathbb{R}$, the Bessel function $J_{\nu}$ satisfies the following differential
equation
$$
x^2 y'' + x y' + (x^2 - \nu^2) y = 0 \qquad  x\in (0, + \infty).
$$
Besides, the function $J_{\nu}$ has an infinite number of real zeros which are simple (see \cite{Elbert2001}). We denote by $(j_{\nu,n})_{n\geq 1}$ the strictly
increasing sequence of the positive zeros of $J_{\nu}$:
$$ j_{\nu,1} < j_{\nu,2} < \cdots<j_{\nu,n}< \cdots$$
and we recall that
$$j_{\nu,n} \rightarrow +\infty\quad \text{as}\quad n\rightarrow +\infty$$
and the following bounds on the zeros $j_{\nu, n}$, which are provided in \cite{LM2008}:
\begin{align}
&\forall \nu \in \Big(0, \dfrac{1}{2}\Big],\, \forall n \geq 1, &\big( n + \frac{\nu}{2} - \frac{1}{4} \big) \pi \leq j_{\nu, n } \leq \big( n + \frac{\nu}{4} - \frac{1}{8} \big) \pi, \label{boundinf12}\\
&\forall \nu \geq \dfrac{1}{2},\, \forall n \geq 1,  &\big( n + \frac{\nu}{4} - \frac{1}{8} \big) \pi \leq j_{\nu, n } \leq \big( n + \frac{\nu}{2} - \frac{1}{4} \big) \pi.\label{boundsup12}
\end{align}
Moreover, we have the following results (see \cite[Proposition 7.8]{KL2005}):
\begin{lemma}\label{lemmadifference}
Let $j_{\nu,n}, n\geq 1$ be the positive zeros of the Bessel function $J_{\nu}$. Then, the
following holds:
\begin{itemize}
\item[$\bullet$] The difference sequence ($ j_{\nu, n+1} - j_{\nu, n})_{n}$  converges to $\pi$ as $n \longrightarrow +\infty.$
\item[$\bullet$] The sequence $(j_{\nu, n+1} - j_{\nu, n})_{n}$ is strictly decreasing if $|\nu| > \frac{1}{2}$, strictly
increasing if $|\nu| < \frac{1}{2}$, and constant if $\nu= \frac{1}{2}$.
\end{itemize}
\end{lemma}
We also have that the Bessel functions enjoy the following integral formula (see \cite{Watson}):
$$
\int_0^1 x  J_{\nu} (j_{\nu, n} x)  J_{\nu}(j_{\nu, m} x)\,dx = \frac{\delta_{nm}}{2} [J_{\nu}^{'}(j_{\nu, n})]^2, \quad n,m \in \mathbb{N}^{*},
$$
where, $\delta_{nm}$ is the Kronecker delta.

Now, we give the expression of the spectrum, i.e., the nontrivial solutions $(\lambda, \Phi)$ of
\begin{equation}\label{eigenvaluesproblem}
\begin{cases}
- (x^{\alpha} \Phi^{'}(x))^{'} = \lambda \Phi(x),  \quad x \in (0,1), \\
\Phi(1)=0,\\
\begin{cases}
\Phi(0)=0, \quad \text{in the (WD) case} \\
(x^{\alpha}\Phi_x) (0)=0, \quad \text{in the (SD) case}.
\end{cases}
\end{cases}
\end{equation}
From now on, let
\begin{equation*}
\kappa_{\alpha}= \frac{2-\alpha}{2} >0 \quad \text{and} \quad \nu_{\alpha} = \begin{cases}
\frac{1- \alpha }{2-\alpha} \in \Big(0,\frac{1}{2}\Big], \quad \text{in the (WD) case}\\[0.2cm]
\frac{\alpha -1 }{2-\alpha} \geq 0, \quad \text{in the (SD) case}.
\end{cases}
\end{equation*}
Then, one has (see \cite{Gueye,moyano}):
\begin{proposition}
The admissible eigenvalues $\lambda$ for problem \eqref{eigenvaluesproblem} are given by
\begin{equation}\label{eigenvalues}
\lambda_{\nu_{\alpha}, n} = \kappa_{\alpha}^2 j_{\nu_{\alpha}, n}^2, \qquad  \forall n \geq 1.
\end{equation}
and the associated normalized (in $L^2(0,1)$) eigenfunctions takes the form
\begin{equation}\label{eigenfunctions}
\Phi_{\nu_{\alpha}, n}(x) =  \frac{\sqrt{2 \kappa_{\alpha}}}{|J'_{\nu_{\alpha}}(j_{\nu_{\alpha}, n})|}
x^{\frac{1-\alpha}{2}} J_{\nu_{\alpha}} (j_{\nu_{\alpha}, n} x^{\kappa_{\alpha}}), \qquad x\in(0,1), \quad n \geq 1.
\end{equation}
Moreover, the family $(\Phi_{\nu_{\alpha}, n})_{n\geq 1} $ forms an orthonormal basis of $L^2(0,1)$.
\end{proposition}

In both cases of degeneracy, the spectrum of the
associated degenerate operator satisfies the following properties proved for the general degenerate/singular operator spectrum in \cite{AHSS2021}.

\begin{lemma}\label{gap-cv-series}
Let $(\lambda_{\nu_{\alpha}, k})_{k\geq 1}$ be the sequence of eigenvalues of the spectral problem \eqref{eigenvaluesproblem}.
Then, the following properties hold:
\begin{enumerate}
\item For all $n,m \in \mathbb{N}^\star$, there is a constant $\rho_{\alpha} > 0$ such that the sequence of eigenvalues
$(\lambda_{\nu_{\alpha}, n})_{n\geq 1}$ satisfy the gap condition:
\begin{equation}\label{gapresult}
| \lambda_{\nu_{\alpha}, n} - \lambda_{\nu_{\alpha}, m} | \geq \rho_{\alpha} |n^2 - m^2|, \qquad \forall n,m \geq 1.
\end{equation}
\item The series $\displaystyle\sum_{n}  \frac{1}{\lambda_{\nu_{\alpha},n}}$ is convergent.
\end{enumerate}
\end{lemma}

\subsection{Spectral properties of vectorial degenerate operators}

Let us consider the degenerate vectorial operators
\begin{equation}\label{operatorL}
L:=\left(\begin{matrix}
- \partial_x (x^{\alpha} \partial_x \cdot ) & 0 \\
0 & - \partial_x (x^{\alpha} \partial_x \cdot )
\end{matrix}\right) - A: D(L) \subset L^2(0,1)^2 \rightarrow L^2(0,1)^2
\end{equation}
and also its adjoint
\begin{equation*}
L^*:=\left(\begin{matrix}
- \partial_x (x^{\alpha} \partial_x \cdot ) & 0 \\
0 & - \partial_x (x^{\alpha} \partial_x \cdot )
\end{matrix}\right) - A^{*}
\end{equation*}
with domains $D(L)=D(L^*)=H^{2}_{\alpha}(0,1)^2$.
This section is devoted to presenting some spectral properties of the operators \( L \) and \( L^* \), which will be useful for developing the moment method. In particular, we state the following result:
\begin{proposition}
\begin{enumerate}
\item
The spectra of $L$ and $L^*$ are given by
\begin{equation}\label{sigmaLL*}
\sigma(L) = \sigma(L^*) =
\big\{ \lambda_{\nu_{\alpha}, n}^{(1)} , \lambda_{\nu_{\alpha}, n}^{(2)} \big\}_{n\geq1}
= \big\{ \kappa_{\alpha}^2 j_{\nu_{\alpha},n}^2 - \mu_1 , \;  \kappa_{\alpha}^2 j_{\nu_{\alpha},n}^2 - \mu_2   \big\}_{n\geq1}
\end{equation}
where $\mu_1$ and $\mu_2$ are the eigenvalues of the matrix $A$ defined by :
\begin{itemize}
\item[$\bullet$] Case 1: $a_2^2 + 4a_1>0 $,
\begin{equation}\label{mu1mu2real}
\mu_1 = \frac{1}{2} \Big(a_2 - \sqrt{a_2^2 + 4a_1}\Big)  \quad \text{and} \quad \mu_2 = \frac{1}{2} \Big(a_2 + \sqrt{a_2^2 + 4a_1}\Big),
\end{equation}
\item[$\bullet$] Case 2: $a_2^2 + 4a_1<0 $,
\begin{equation}\label{mu1mu2complex}
\mu_1 = \frac{1}{2} \Big(a_2 + i \sqrt{-(a_2^2 + 4a_1)}\Big)  \quad \text{and} \quad \mu_2 = \frac{1}{2}  \Big(a_2 - i \sqrt{-(a_2^2 + 4a_1)} \Big),
\end{equation}
\end{itemize}
\item
For each $n \geq 1$, the corresponding eigenfunctions of $L$ (resp., $L^*$) associated to $\lambda_{\nu_{\alpha}, n}^{(1)}$ and $\lambda_{\nu_{\alpha}, n}^{(2)}$ are respectively given by
\begin{equation}\label{eigenf-L}
\psi_n^{(1)}  =U_{1} \Phi_{\nu_{\alpha}, n}, \qquad \psi_n^{(2)}=U_{2} \Phi_{\nu_{\alpha}, n},
\end{equation}
with
\begin{equation*}
U_{1}=\frac{1}{\mu_{1}-\mu_{2}}\left(\begin{array}{c}-\mu_{2} \\
1\end{array}\right) \quad \text { and } \quad U_{2}=\frac{1}{\mu_{2}-\mu_{1}}\left(\begin{array}{c}
-\mu_{1} \\1
\end{array}\right)
\end{equation*}
$\left(\text{resp}.,\right.$
\begin{equation}\label{eigenf-L*}
\Psi_n^{(1)}  = V_{1}\Phi_{\nu_{\alpha}, n}, \qquad \Psi_n^{(2)}  = V_{2}\Phi_{\nu_{\alpha}, n},
\end{equation}
with
\begin{equation*}
V_{1}=\left(\begin{array}{c}
1 \\
\mu_{1}
\end{array}\right) \quad \text { and } \quad V_{2}=\left(\begin{array}{c}
1 \\
\mu_{2}
\end{array}\right)\left.\right).
\end{equation*}
\end{enumerate}
\label{propertieseigenfLL}
\end{proposition}
As mentioned earlier, the existence of the biorthogonal family is essential for expressing the control force as a formal solution to the moment problem. To this end, the following proposition is crucial for verifying that the sequence of eigenvalues of \( L \) and \( L^* \) satisfies the conditions in Theorem \ref{thm-biorth}. Since its proof closely follows that in \cite{AHSS2021, ASS2022}, we omit it.
\begin{proposition}\label{prop-incres-seq} 
Assume that condition \eqref{spectre-cond} holds. Then, the family 
\begin{equation}\label{increas-sequence}
\begin{aligned}
\big\{ \Lambda_{\nu_{\alpha}, n} \big\}_{n\geq1} &= \{\lambda_{\nu_{\alpha}, k} + \mu_2 - \mu_1: k\geq 1\} \cup \{\lambda_{\nu_{\alpha}, k}: k \geq 1\}  \\
&=\big\{ \lambda_{\nu_{\alpha}, n}^{(1)} + \mu_2 , \lambda_{\nu_{\alpha}, n}^{(2)}+\mu_2 \big\}_{n\geq1},
\end{aligned}
\end{equation}
satisfies the hypotheses in Theorem \ref{thm-biorth}.
\end{proposition}
Furthermore, let us present the following result on the set of eigenfunctions of the operators $L$ and
$L^*$. It reads as follows:
\begin{proposition}\label{rieszbasisresult}
Let us consider the sequences
\begin{equation}\label{BB*}
\mathcal{B} = \big\{\psi_n^{(1)} , \psi_n^{(2)}, \quad n \geq 1 \big\} \quad \text{and} \quad
\mathcal{B}^* = \big\{\Psi_n^{(1)} , \Psi_n^{(2)}, \quad n \geq 1 \big\}.
\end{equation}
Then,
\begin{enumerate}
\item $\mathcal{B}$ and $\mathcal{B}^*$ are biorthogonal families in $L^2 (0, 1)^2$.
\item $\mathcal{B}$ and $\mathcal{B}^*$ are complete sequences in $L^2 (0, 1)^2$.
\item The sequences $\mathcal{B}$ and $\mathcal{B}^*$ are biorthogonal Riesz bases of $L^2 (0, 1)^2$.
\item The sequence $\mathcal{B}^*$ is a basis of $H^{1}_{\alpha}(0,1)^2$ and $\mathcal{B}$ is its biorthogonal basis in $H^{-1}_{\alpha}(0,1)^2$.
\end{enumerate}
\end{proposition}
For the proofs of Propositions \ref{propertieseigenfLL} and \ref{rieszbasisresult}, we refer to Appendix \ref{appendix:a}.

Throughout this paper, we use the following notation. Given \( z \in \mathbb{C} \), \( \Re(z) \) and \( \Im(z) \) denote the real and imaginary parts of \( z \), respectively.

\section{Boundary controllability of the linear system \eqref{bound-sys}}\label{Section-approx}
This section is devoted to providing null controllability results for system \eqref{bound-sys}. To this end, it is important to highlight the key technical tools used to address this problem. By transforming the controllability problem into a moment problem, the existence of a biorthogonal family will play a crucial role in deriving an explicit control function, as well as in obtaining an estimate of the control cost. The following theorem establishes a connection between the presence of biorthogonal sequences to complex exponentials, along with the gap condition between eigenvalues. See, for instance, \cite{AHSS2021, ASS2022, ABGT2011, FBT}. In particular,  one has
\begin{theorem}[\cite{BBGBO14}]\label{thm-biorth}
Let $ \{\Lambda_n\}_{n\geq1} $ be a sequence of complex numbers fulfilling the following assumptions:
\begin{enumerate}
\item $\Lambda_n \neq \Lambda_m$ for all $n, m \geq 1$ with $n \neq m$;
\item $\Re (\Lambda_n) > 0$ for every $n \geq 1$;
\item for some $ \delta >0$
$$|\Im (\Lambda_n)| \leq \delta \sqrt{\Re (\Lambda_n)} \quad \forall n \geq 1;$$
\item $ \{\Lambda_n\}_{n\geq1} $ is nondecreasing in modulus,
$$ |\Lambda_n| \leq |\Lambda_{n+1}| \quad \forall n \geq 1; $$
\item $ \{\Lambda_n\}_{n\geq1} $ satisfies the following gap condition: for some $\varrho, q >0 $,
\begin{equation}\label{strong-gap}
\left\{
\begin{array}{l}
|\Lambda_n - \Lambda_m| \geq \varrho |n^2 - m^2| \quad \forall n, m: |n -m|\geq q, \\
\inf\limits_{n\neq m, \; |n -m|<q}|\Lambda_n - \Lambda_m| > 0;
\end{array}\right.
\end{equation}
\item for some $p, s> 0,$
\begin{equation}\label{ineq-counting}
| p \sqrt{r} - \mathcal{N}(r)| \leq s, \quad \forall r >0,
\end{equation}
where $\mathcal{N}$ is the counting function associated with the sequence
$ \{\Lambda_n\}_{n\geq1} $, that is the function defined by
\begin{equation*}
\mathcal{N}(r) = \#\{ n : \, |\Lambda_n| \leq r \}, \qquad \forall r>0.
\end{equation*}
\end{enumerate}
Then, there exists $\widetilde{T}_0 > 0 $, such that for any $T\in (0,\widetilde{T}_0)$, we can find a family $\{q_n \}_{n\geq 1} \subset L^2(-T/2, T/2) $ biorthogonal to $\{e^{- \Lambda_n t}\}_{n\geq 1}$ i.e., a family $\{q_n \}_{n\geq 1} $ in $L^2(-T/2, T/2)$ such that
\begin{align}\label{biorthogonality}
\int_{-T/2}^{T/2} q_n(t) e^{- \Lambda_m t} \, dt = \delta_{nm}.
\end{align}
Moreover, there exists a positive constant $ C > 0 $  independent of $T$ for which
\begin{equation}\label{boundqni}
\|q_n \|_{L^2(-T/2, T/2)} \leq C e^{C\sqrt{\Re (\Lambda_n)} + \frac{C}{T}}, \qquad \forall n \geq 1.
\end{equation}
\end{theorem}
 
As it is well known, the controllability of system \eqref{bound-sys} can be characterized in terms of appropriate properties of the solutions of its corresponding homogeneous adjoint problem (see for instance \cite[Theorem. 2.1]{ABGT2011}) defined as
\begin{equation}\label{adjoint system-homog}
\left\{\begin{array}{lll}
-\varphi_{t}-( x^{\alpha} \varphi_x )_x =A^{*}\varphi & & \text { in } Q_{T} \\
\varphi(t, 1)=0 & & t \in (0, T)\\
\left\{\begin{array}{l}
\varphi(t, 0)=0, \quad 0 \leq \alpha<1 \\
x^{\alpha} \varphi_{x}(t, 0)=0, \quad 1 \leq \alpha<2
\end{array}\right. & & \text { on }(0, T) \\
\varphi(T,\cdot)=\varphi_{0} & & \text { in }(0, 1).
\end{array}\right.
\end{equation}
Note that system \eqref{adjoint system-homog} above is in particular the adjoint system \eqref{adjoint system} with $g=(0,0)$.

In order to provide these characterizations, we use the following result which gives a relation between the solutions of systems \eqref{bound-sys} and \eqref{adjoint system-homog} (For a proof, see for instance \cite{FBT}).
\begin{proposition}\label{porp dual and control}
Let $B$ the matrix given by \eqref{Vec-sys-matrix}. Let us consider $y_0 \in H_\alpha^{-1}(0,1)^2$, $v \in L^2(0, T)$ and $\varphi_0 \in H_\alpha^1(0,1)^2$. Then, the solution $y$ of system \eqref{bound-sys} associated to $y_0$ and $v$, and the solution $\varphi$ of the adjoint system \eqref{adjoint system-homog} associated to $\varphi_0$ satisfy
$$
\int_0^T v(t) B^*\left(x^\alpha \varphi_x\right)(t, 1) d t=-\left\langle y(T), \varphi_0\right\rangle_{H_\alpha^{-1}, H_\alpha^1}+\left\langle y_0, \varphi(0, \cdot)\right\rangle_{H_\alpha^{-1}, H_\alpha^1}.
$$
\end{proposition}
We shall now proceed to prove Theorem \ref{Thm-null-cont1}, which present the first main result.

\begin{proof}[Proof of Theorem \ref{Thm-null-cont1}] To prove Theorem \ref{Thm-null-cont1}, we transform the controllability problem into a moment problem. Using Proposition \ref{porp dual and control}, we deduce that the control $v\in L^2(0,T)$ drives the solution of \eqref{bound-sys} to zero at time $T$ if and only if $v \in L^2(0,T)$ satisfies
\begin{equation}\label{firststep}
\int_0^T B^* (x^{\alpha} \varphi_x) (t,1) \, v(t) \; dt = \langle y_0 , \varphi(0, \cdot) \rangle_{H_{\alpha}^{-1}, H_{\alpha}^{1}}, \qquad \forall \varphi_0 \in H_{\alpha}^1(0, 1)^2
\end{equation}
where $\varphi \in  C^0\big([0, T];  H_{\alpha}^1(0,1)^2 \big) \cap L^2\big(0, T; H^{2}_{\alpha}(0,1)^2\big)$ is the solution of the adjoint system \eqref{adjoint system-homog} associated to $\varphi_0$.

Observe that, using Proposition \ref{rieszbasisresult}, the corresponding solution $\varphi$ of system \eqref{adjoint system-homog} associated to $\varphi_0$ is given by
\begin{equation*}
\varphi(t,x)= \sum_{k\geq 1} \Big(\langle \psi_k^{(1)}, \varphi_0 \rangle_{H_{\alpha}^{-1},H_{\alpha}^1}  \Psi_{k}^{(1)}e^{-\lambda_{\nu_{\alpha},k}^{(1)}(T-t)}
+
\langle \psi_k^{(2)}, \varphi_0 \rangle_{H_{\alpha}^{-1},H_{\alpha}^1}  \Psi_{k}^{(2)} e^{-\lambda_{\nu_{\alpha},k}^{(2)}(T-t)}\Big).
\end{equation*}
Since $\mathcal{B}^*$ is a basis for $H_{\alpha}^1(0,1)^2$, we find that $\varphi(t,x) = \Psi_n^{(i)}(x) e^{-\lambda_{\nu_{\alpha}, n}^{(i)} (T-t)}$ is the solution of system \eqref{adjoint system-homog} associated with $\varphi_0= \Psi_n^{(i)}$.
Therefore, we can deduce that the identity \eqref{firststep} is equivalent to
$$
\int_0^T B^* (x^{\alpha} \Psi_{n,x}^{(i)}) (1) v(t) e^{-\lambda_{\nu_{\alpha}, n}^{(i)} (T-t)} dt =  e^{-\lambda_{\nu_{\alpha}, n}^{(i)} T} \langle y_0 , \Psi_n^{(i)} \rangle_{H_{\alpha}^{-1}, H_{\alpha}^{1}}, \quad  \forall n\geq 1, \quad i=1,2.
$$
Taking
into account the expressions of $\Psi_n^{(i)}$ given by \eqref{eigenf-L*}, we infer that $v\in L^2(0,T)$ is
a null control for system \eqref{bound-sys} associated to $y_0$ if and only if
\begin{align*}
{\small \frac{\sqrt{2} \kappa_{\alpha}^{\frac{3}{2}} j_{\nu_{\alpha}, n} }{|J'_{\nu_{\alpha}}(j_{\nu_{\alpha}, n})|}  J'_{\nu_{\alpha}} (j_{\nu_{\alpha}, n} ) B^* V_{i} \int_0^T v(t) e^{-\lambda_{\nu_{\alpha}, n}^{(i)} (T-t)} dt  =  e^{-\lambda_{\nu_{\alpha}, n}^{(i)} T} \langle y_0 , \Psi_n^{(i)} \rangle_{H_{\alpha}^{-1}, H_{\alpha}^{1}}, \quad  \forall n\geq 1, \quad i=1,2}
\end{align*}
and equivalently, 
\begin{equation}\label{moment}
\int_0^T v(t) e^{-\lambda_{\nu_{\alpha}, n}^{(i)} (T-t)} dt = C_{\nu_{\alpha},n}^{(i)} , \quad \forall n\geq 1, \quad i=1,2,
\end{equation}
where
$C_{\nu_{\alpha},n}^{(i)}$ is given by
\begin{equation}\label{expres1}
C_{\nu_{\alpha},n}^{(i)}=  \frac{ |J'_{\nu_{\alpha}}(j_{\nu_{\alpha}, n})| e^{-\lambda_{\nu_{\alpha}, n}^{(i)} T} }{\sqrt{2} \kappa_{\alpha}^{\frac{3}{2}} j_{\nu_{\alpha}, n}  J'_{\nu_{\alpha}}(j_{\nu_{\alpha}, n} ) B^* V_{i}}  \langle y_0 , \Psi_n^{(i)}\rangle_{H_{\alpha}^{-1}, H_{\alpha}^{1}}, \quad  \forall n\geq 1, \quad i=1,2.
\end{equation}
At this stage, the strategy for solving the moment problem \eqref{moment} is to use the concept of a biorthogonal family. In fact, Proposition \ref{prop-incres-seq} and Theorem \ref{thm-biorth} guarantee the existence of \( \widetilde{T}_0 > 0 \), such that for any \( T \in (0, \widetilde{T}_0) \), there exists a biorthogonal family \( \{q_n^{(1)}, q_n^{(2)}\}_{n \geq 1} \) to \( \{e^{-(\lambda_{\nu_{\alpha}, n}^{(1)} + \mu_2) t}, e^{-(\lambda_{\nu_{\alpha}, n}^{(2)} + \mu_2) t}\}_{n \geq 1} \) in \( L^2(-T/2, T/2) \) which also satisfies:
\begin{equation}\label{boundqni2}
\|q_n^{(i)} \|_{L^2(-T/2,T/2)} \leq C e^{\sqrt{\mathcal{R} (\lambda_{\nu_{\alpha}, n}^{(i)}+\mu_2)} + \frac{C}{T}}, \qquad \forall n \geq 1, \quad i=1,2.
\end{equation}
for some positive constant $C$ independent of $T$.\\
Performing the change of variable \( s = T/2 - t \) in \eqref{moment}, the controllability problem is then reduced to the following moment problem:
Given $y_0 \in  H_{\alpha}^{-1}(0,1)^2$ find $v\in L^2(0,T)$ such that $u(s) = v(T/2- s) e^{\mu_2s} \in L^2(-T/2,T/2)$ satisfies
\begin{equation}\label{moment1}
\int_{-T/2}^{T/2} u(s) e^{-(\lambda_{\nu_{\alpha}, n}^{(i)}+\mu_2)s} ds = \widehat{C}_{\nu_{\alpha},n}^{(i)} , \quad \forall n\geq 1, \quad i=1,2,
\end{equation}
with $\widehat{C}_{\nu_{\alpha},n}^{(i)} = e^{\lambda_{\nu_{\alpha}, n}^{(i)} T/2} C_{\nu_{\alpha},n}^{(i)}$. Then, a formal solution to the moment problem \eqref{moment1} is given by
$$u(s) = \sum_{n\geq1} ( \widehat{C}_{\nu_{\alpha},n}^{(1)} q_n^{(1)}(s)  +  \widehat{C}_{\nu_{\alpha},n}^{(2)} q_n^{(2)}(s) ).$$
Thus,
\begin{equation}\label{nullcontrol_system}
v(s)= \sum_{n\geq1}  \big( \widehat{C}_{\nu_{\alpha},n}^{(1)} q_n^{(1)}(T/2- s)  +  \widehat{C}_{\nu_{\alpha},n}^{(2)} q_n^{(2)}(T/2- s) \big) e^{-\mu_2(T/2- s)}.
\end{equation}
The only remaining point is to prove that $v \in L^2(0,T)$. This comes directly from the estimate \eqref{boundqni2} and the fact that
\begin{equation*}
\| \Psi_n^{(i)}\|_{ H_{\alpha}^{1}}=\|V_{i} \Phi_{\nu_{\alpha},n}\|_{ H_{\alpha}^{1}}  \leq C \sqrt{\lambda_{\nu_{\alpha}, n}}= C \kappa_{\alpha} j_{\nu_{\alpha}, n},\quad \forall n\geq 1,\quad i=1,2,
\end{equation*}
for some positive constant $C$. Indeed, the previous inequality implies
\begin{align*}
|\langle y_0 , \Psi_n^{(i)}\rangle_{H_{\alpha}^{-1}, H_{\alpha}^{1}}|  \leq  \|y_0 \|_{H_{\alpha}^{-1}}\|\Psi_n^{(i)}\|_{ H_{\alpha}^{1}} \leq C \kappa_{\alpha} j_{\nu_{\alpha}, n} \|y_0 \|_{ H_{\alpha}^{-1}},\quad  \forall n\geq 1,\quad i=1,2.
\end{align*}
Moreover, from \eqref{expres1}, by using the expressions of $\kappa_{\alpha}$, we obtain
\begin{equation}\label{estinCnu}
|\widehat{C}_{\nu_{\alpha},n}^{(i)}| \leq \frac{C}{\sqrt{2-\alpha}} e^{-\lambda_{\nu_{\alpha}, n}^{(i)} T/2}  \|y_0 \|_{ H_{\alpha}^{-1}},\quad  \forall n\geq 1,\quad i=1,2.
\end{equation}
Now, taking into account the definition of $\lambda_{\nu_{\alpha}, n}$, we get for a new constants $C$ not depending on $n$ and $T$
$$|e^{-\lambda_{\nu_{\alpha}, n}^{(i)} T/2} | \leq e^{CT} e^{-\lambda_{\nu_{\alpha}, n} T/2}\quad $$
and
\begin{equation}
\sqrt{\mathcal{R}(\lambda_{\nu_{\alpha}, n}^{(i)}+\mu_2)} \leq  C\sqrt{\lambda_{\nu_{\alpha}, n}}, \quad \forall n \geq 1.
\end{equation}
Coming back to the expression \eqref{nullcontrol_system} of the null control $v$ and using the previous estimates, we get
\begin{equation}\label{step2}
\|v\|_{L^2(0,T)} \leq  \frac{C e^{CT}}{\sqrt{2-\alpha}} \|y_0 \|_{ H_{\alpha}^{-1}} \sum_{n\geq1}  e^{-\lambda_{\nu_{\alpha}, n} T/2} e^{C\sqrt{\lambda_{\nu_{\alpha}, n}} + \frac{C}{T}}.
\end{equation}
Using Young's inequality,
$$ C\sqrt{\lambda_{\nu_{\alpha}, n}} \leq \frac{\lambda_{\nu_{\alpha}, n}T}{4}+ \frac{C^2}{T}$$
we see that
\begin{equation*}
\|v\|_{L^2(0,T)} \leq  \frac{C e^{C T+ \frac{C}{T}}}{\sqrt{2-\alpha}} \|y_0 \|_{ H_{\alpha}^{-1}} \sum_{n\geq1}  e^{-\lambda_{\nu_{\alpha}, n} T/4}.
\end{equation*} 
On the other hand, by \eqref{boundinf12} and \eqref{boundsup12}, it can be easily checked that there exist a constant  $C> 0$ such that
$$
C \kappa_{\alpha}^2 n^2 \leq  \lambda_{\nu_{\alpha}, n} = \kappa_{\alpha}^2 j_{\nu_{\alpha}, n}^2, \quad \forall n \geq 1.
$$
Finally,
\begin{align*}
\|v\|_{L^2(0,T)} &\leq  \frac{C}{\sqrt{2-\alpha}} e^{C T+ \frac{C}{T}} \|y_0 \|_{ H_{\alpha}^{-1}} \sum_{n\geq1}  e^{-C \kappa_{\alpha}^2 n^2 T} \\
&\leq  \frac{C}{\sqrt{2-\alpha}} e^{C T+ \frac{C}{T}} \|y_0 \|_{ H_{\alpha}^{-1}} \int_{0}^{\infty} e^{-C \kappa_{\alpha}^2 T s^2}\,ds \\
&= \frac{C}{(2-\alpha)^{\frac{3}{2}}} e^{C T+ \frac{C}{T}} \|y_0 \|_{ H_{\alpha}^{-1}} \sqrt{\frac{\pi}{T}}\\
&\leq C e^{C T+ \frac{C}{T}} \|y_0 \|_{ H_{\alpha}^{-1}},
\end{align*}
where $C$ independent of $T$. This inequality shows that
$v \in L^2(0, T)$ and yields the desired estimate on the null control in the case where $T<\widetilde{T}_0$. On the other hand, when $T\geq \widetilde{T}_0$, actually reduced to the previous one. Indeed, any continuation by zero of a control on $\left(0, \widetilde{T}_0 / 2\right)$ is a control on $(0, T)$ and the estimate follows from the decrease of the cost with respect to the time. To be more precise, we define the operator 
\begin{align}\label{operator C0}
\begin{aligned}
\mathcal{C}_{T}^{(0)}: H^{-1}_\alpha\left(0, 1 ; \mathbb{R}^{2}\right) &\rightarrow L^{2}(0, T)\\
y_0 &\mapsto \mathcal{C}_{T}^{(0)}\left(y_{0}\right):=\sum_{k \geq 1}\left(c_{k 1}\left(y_{0}\right) q_{k}^{(1)}(T-\cdot)+c_{k 2}\left(y_{0}\right) q_{k}^{(2)}(T-\cdot)\right), 
\end{aligned}
\end{align}

it suffices to set the null control function to $0$ for the time interval $(\widetilde{T}_0/2,T )$. Indeed, $v$ is given by
$$
v(t)=\mathcal{C}_{T}^{(0)}\left(y_{0}\right)(t):= \begin{cases}\mathcal{C}_{\widetilde{T}_{0} / 2}^{(0)}\left(y_{0}\right)(t) & \text { if } t \in\left[0, \frac{\widetilde{T}_{0}}{2}\right] \\[.1cm] 0 & \text { if } t \in\left[\frac{\widetilde{T}_{0}}{2}, T\right] .\end{cases}
$$
Clearly $\mathcal{C}_{T}^{(0)} \in \mathcal{L}\left(H^{-1}_\alpha\left(0, 1 ; \mathbb{R}^{2}\right), L^{2}(0, T)\right)$ and consequently, the following estimate follows
\begin{equation}
\|v\|_{L^2(0,T)} = \left\|\mathcal{C}_{T}^{(0)}\left(y_{0}\right)\right\|_{L^{2}(0, T)} \leq C_0 e^{\frac{2M}{\widetilde{T}_0}} \|y_0 \|_{ H_{\alpha}^{-1}}=C_{1}\left\|y_{0}\right\|_{H^{-1}_\alpha}.
\end{equation}
So, we can conclude \eqref{OperatorC0bound} for a new positive constant $C_{0}$. This completes the proof of Theorem \ref{Thm-null-cont1}.
\end{proof}

\section{Boundary controllability of the non linear system}\label{Section-null}
In this section, we address the main result of this work, namely the controllability of the nonlinear system \eqref{bound-nonlinear-sys}. To this end, we employ a fixed-point strategy that fundamentally relies on a null controllability result for the non-homogeneous linear system \eqref{bound-nonhomog-sys}, with a source term \( f \in L^{2}\left(0, 1 ; \mathbb{R}^{2}\right) \) given in an appropriate weighted Lebesgue space; see \eqref{weighted normed spaces}.
\subsection{Null controllability of the non-homogeneous system \eqref{bound-nonhomog-sys}}
As said before, our next objective will be to show a null controllability result for \eqref{bound-nonhomog-sys} when $y_{0} \in H^{-1}_\alpha(0,1)^2$ and $f$ is a given function satisfying appropriate assumptions.

To be more precise, we will give the null controllability result of the following system (given before in \eqref{bound-nonhomog-sys}), based on some ideas from \cite{Burgos_2020, Liu_2012}:
\begin{equation}
\left\{
\begin{array}{lll}
\partial_{t}y - ( x^{\alpha}  y_x )_x - A y = f ,  &  & \text{in} \;  Q_T:=(0, \, T)\times(0, 1), \\
y(t,1)= B v(t), & & \text{in} \; (0, \, T),  \\
\begin{cases}
& y(t, 0) = 0,   \qquad   \, 0 \leq \alpha < 1 \\
&  x^{\alpha} y_x  (t, 0)= 0 , \quad 1\leq \alpha < 2 \\
\end{cases}  & & t \in (0, T),\\
y(0,x)= y_{0}(x),  & & \text{in}  \;  (0, 1),
\end{array}
\right.
\end{equation}

Thanks to Theorem \ref{Thm-null-cont1}, we obtain an estimate for the cost of the null control of system \eqref{bound-sys}, which is defined by 
$$
\mathcal{K}(T) =\sup\limits_{\| y_0 \|_{H^{-1}_{\alpha}}=1}\left( \inf\limits_{ v \text{ admissible } }\| v \|_{L^2(0,T)} \right),\quad \forall T>0.
$$
One has
$$
\mathcal{K}(T) \leq C_{0} e^{\frac{M}{T}}, \quad \forall T>0
$$
with $C_{0}$ and $M$ two positive constants only depending on the parameters of system \eqref{bound-sys}.

In order to obtain a null controllability result for the non-homogeneous problem \eqref{bound-nonhomog-sys} at time $T>0$, let us introduce the following continuous non increasing function
\begin{align}
\gamma(t):=e^{\frac{M}{t}}, \forall t>0
\end{align}
and, for $t \in[0, T]$
\begin{align}\label{weight functions}
\rho_{\mathcal{F}}(t):=e^{-\frac{q^{2}(p+1) M}{(q-1)(T-t)}}, \quad \rho_{0}(t):=e^{-\frac{p M}{(q-1)(T-t)}}, \quad \forall t \in\left[T\left(1-\frac{1}{q^{2}}\right), T\right].
\end{align}
\begin{remark}
The above two functions $\rho_{\mathcal{F}}$ and $\rho_{0}$ given in \eqref{weight functions}, extended into $\left[0, T\left(1-1 / q^{2}\right)\right]$ in a constant way, where $p, q>1$ are two constants that will be well chosen later, for a technical role in the next proofs. Notice that $\rho_{\mathcal{F}}$ and $\rho_{0}$ are continuous and non increasing functions in $[0, T]$ and $\rho_{\mathcal{F}}(T)=\rho_{0}(T)=0$.
\end{remark}
Before presenting the null controllability result for the system \eqref{bound-nonhomog-sys}, we associate to the functions $\rho_{\mathcal{F}}$ and $\rho_{0}$ the Banach space $\mathcal{H}$ and the Hilbert spaces $\mathcal{F}, \mathcal{V}$ and $\mathcal{H}_{0}$ defined by :
\begin{align}\label{weighted normed spaces}
\mathcal{F}:&=\left\{f \in L^{2}\left(Q_{T} ; \mathbb{R}^{2}\right): \frac{1}{\rho_{\mathcal{F}}}f \in L^{2}\left(Q_{T} ; \mathbb{R}^{2}\right)\right\}, \\
\mathcal{V}:&=\left\{v \in L^{2}(0, T): \frac{1}{\rho_{0}}v \in L^{2}(0, T)\right\}, \\
\mathcal{H}_{0}:&=\left\{y \in L^{2}\left(Q_{T} ; \mathbb{R}^{2}\right): \frac{1}{\rho_{0}} y \in L^{2}\left(Q_{T} ; \mathbb{R}^{2}\right)\right\} ,\\
\mathcal{H} :&=\left\{y \in L^{2}\left(Q_{T} ; \mathbb{R}^{2}\right): \frac{1}{\rho_{0}}y \in L^{2}\left(Q_{T}\right) \times C^{0}\left(\overline{Q}_{T}\right)\right\} .
\end{align}

The inner product in $\mathcal{F}$ is given by
\begin{align}\label{scalar prdct}
\left(f_{1}, f_{2}\right)_{\mathcal{F}}:=\iint_{Q_{T}} \rho_{\mathcal{F}}^{-2}(t) f_{1}(t, x) \cdot f_{2}(t, x) d x d t, \quad \forall f_{1}, f_{2} \in \mathcal{F},
\end{align}
where the notation $\displaystyle f_1 . f_2$ represents the scalar product in $\mathbb{R}^2$. While a similar definition as in \eqref{scalar prdct}, can be made for $(\cdot, \cdot)_{\mathcal{V}}$ and $(\cdot, \cdot)_{\mathcal{H}_{0}} .$

On the other hand, $\mathcal{H}$ is endowed with the norm
$$
\|y\|_{\mathcal{H}}:=\left(\left\|y_{1} / \rho_{0}\right\|_{L^{2}\left(Q_{T}\right)}^{2}+\left\|y_{2} / \rho_{0}\right\|_{C^{0}\left(\overline{Q}_{T}\right)}^{2}\right)^{1 / 2}, \quad \forall y=\left(y_{1}, y_{2}\right) \in \mathcal{H}.
$$
It is clear that $\mathcal{F}, \mathcal{V}$ and $\mathcal{H}_{0}$ are Hilbert spaces and $\mathcal{H}$ is a Banach space with the norm $\|.\|_{\mathcal{H}}$.

We are now in a position to prove the main result in this section which is Theorem \ref{thm-operators}.
\begin{proof}[\textbf{Proof of Theorem \ref{thm-operators}}]
 In order to prove Theorem \ref{thm-operators}, we adapt to system  \eqref{bound-nonhomog-sys} a general technique developed in \cite{Liu_2012} that permits to prove a null controllability result for a non-homogenous linear problem from the corresponding controllability result for the homogenous problem. \\

Let us consider $p, q>1$ and $T>0$. We define the sequence
\begin{equation}\label{Tk}
T_{k}=T-\frac{T}{q^{k}}, \quad \forall k \geq 0
\end{equation}
From the previous formula \eqref{Tk} and the definition of $\rho_{0}$ and $\rho_\mathcal{F}$ given in \eqref{weight functions}, one has
\begin{equation}\label{weight functions link}
\rho_{0}\left(T_{k+2}\right)=\rho_{\mathcal{F}}\left(T_{k}\right) e^{\frac{M}{T_{k+2}-T_{k+1}}}, \quad \forall k \geq 0
\end{equation}

Let us now take $y_{0} \in H_{\alpha}^{-1}(0,1)^2$ (resp., $\left.y_{0} \in H^{-1}_{\alpha}(0, 1) \times H_{\alpha}^{1}(0, 1)\right)$ and $f \in \mathcal{F}$.\\ Thus, we introduce the sequence $\left\{a_{k}\right\}_{k \geq 0} \subset H_{\alpha}^{-1}(0,1)^2$ (resp. $\left\{a_{k}\right\}_{k \geq 0} \subset H^{-1}_{\alpha}(0, 1) \times$ $H_{\alpha}^{1}(0, 1)$ if $\left.y_{0} \in H^{-1}_{\alpha}(0, 1) \times H_{\alpha}^{1}(0, 1)\right)$ defined by
\begin{equation}\label{ak}
a_{0}=y_{0}, \quad a_{k+1}=\tilde{y}_{k}\left(T_{k+1}^{-}\right), \quad \forall k \geq 0
\end{equation}
where $\tilde{y}_{k}$ is the solution to the linear system
\begin{equation}\label{linear system y tilde}
\begin{cases}\tilde{y}_{t}-( x^{\alpha}  \tilde{y}_x )_x-A \tilde{y}=f & \text { in }(0, 1) \times\left(T_{k}, T_{k+1}\right):=Q_{k} \\ \tilde{y}(0, \cdot)=\tilde{y}(1, \cdot)=0 & \text { on }\left(T_{k}, T_{k+1}\right) \\ \tilde{y}\left(\cdot, T_{k}^{+}\right)=0 & \text { in }(0, 1)\end{cases}
\end{equation}
where $A$ is the matrix given in \eqref{Vec-sys-matrix}.

From Proposition \ref{prop.2.1}, it is clear that this system admits a unique solution
$$
\tilde{y}_{k} \in L^{2}\left(T_{k}, T_{k+1} ; H^{2}_{\alpha}(0, 1)\right) \cap C^{0}\left(\left[T_{k}, T_{k+1}\right] ; H_{\alpha}^{1}(0,1)^2\right)
$$
which satisfies \eqref{bound phi prop 2.1}. In particular, $\tilde{y}_{k} \in C^{0}\left(\overline{Q}_{k} ; \mathbb{R}^{2}\right)$ and $a_{k+1} \in H_{\alpha}^{1}(0,1)^2$, for any $k \geq 0$, and
\begin{equation}\label{bound y tilde}
\left\|\tilde{y}_{k}\right\|_{C^{0}\left(\overline{Q}_{k} ; \mathbb{R}^{2}\right)}+\left\|a_{k+1}\right\|_{H_{\alpha}^{1}} \leq e^{C T}\|f\|_{L^{2}\left(Q_{k} ; \mathbb{R}^{2}\right)}, \quad \forall k \geq 0
\end{equation}
where $C$ is a positive constant only depending on the coefficients of $A.$
For $k \geq 0$, we also consider the controlled autonomous problem
\begin{equation}\label{linear system y hat}
\left\{\begin{array}{lll}
\hat{y}_{t}- \left(x^\alpha\hat{y}_{x}\right)_x-A \hat{y}=0 & \text { in } Q_{k} \\
\hat{y}(0, \cdot)=0, \quad \hat{y}(1, \cdot)=B v_{k} & \text { on }\left(T_{k}, T_{k+1}\right) \\
\hat{y}\left(\cdot, T_{k}^{+}\right)=a_{k},\quad  \hat{y}\left(\cdot, T_{k+1}^{-}\right)=0 & \text { in }(0, 1)
\end{array}\right.
\end{equation}
where the control $v_{k}$ is given by $v_{k}=\mathcal{C}_{T_{k+1}-T_{k}}^{(0)}\left(a_{k}\right) \in L^{2}\left(T_{k}, T_{k+1}\right)$ (the linear operator $\mathcal{C}_{T_{k+1}-T_{k}}^{(0)}$ is given in Theorem \ref{Thm-null-cont1}). Thanks to Proposition \ref{well.posed.transp}, the solution $\hat{y}_{k}$ of the previous system satisfies
$$
\begin{array}{l}
\hat{y}_{0} \in L^{2}\left(Q_{0} ; \mathbb{R}^{2}\right) \quad\left(\text { resp. }, \hat{y}_{0} \in L^{2}\left(Q_{0} ; \mathbb{R}^{2}\right) \cap C^{0}\left(\left[0, T_{1}\right] ; H^{-1}_{\alpha}(0, 1) \times H_{\alpha}^{1}(0, 1)\right)\right. , \\
 \hat{y}_{k} \in L^{2}\left(Q_{k} ; \mathbb{R}^{2}\right) \cap C^{0}\left(\left[T_{k}, T_{k+1}\right] ; H^{-1}_{\alpha}(0, 1) \times H_{\alpha}^{1}(0, 1)\right), \quad \forall k \geq 1
\end{array}
$$
and, from \eqref{bounds.prop.2.2} (resp., \eqref{bounds.prop.2.2 b}), \eqref{bound y tilde} and Theorem 1.2,
and, for any $k \geq 1$,
$$
\left\|\hat{y}_{k}\right\|_{L^{2}\left(Q_{k}\right) \times C^{0}\left(\overline{Q}_{k}\right)} \leq e^{C T}\left(\left\|a_{k}\right\|_{H^{-1} \times H_{\alpha}^{1}}+\left\|v_{k}\right\|_{L^{2}\left(T_{k}, T_{k+1}\right)}\right) \leq C_{0} e^{C T} e^{\frac{M}{T_{k+1}-T_{k}}}\|f\|_{L^{2}\left(Q_{k} ; \mathbb{R}^{2}\right)}
$$
If we set $Y_{k}:=\tilde{y}_{k}+\hat{y}_{k}$ in $Q_{k}=(0, 1) \times\left(T_{k}, T_{k+1}\right)$, then
$$
\begin{array}{l}
Y_{0} \in L^{2}\left(Q_{0} ; \mathbb{R}^{2}\right) \quad\left(\text { resp. }, Y_{0} \in L^{2}\left(Q_{0} ; \mathbb{R}^{2}\right) \cap C^{0}\left(\left[0, T_{1}\right] ; H^{-1}_{\alpha}(0, 1) \times H_{\alpha}^{1}(0, 1)\right),\right. \\
Y_{k} \in L^{2}\left(Q_{k} ; \mathbb{R}^{2}\right) \cap C^{0}\left(\left[T_{k}, T_{k+1}\right] ; H^{-1}_{\alpha}(0, 1) \times H_{\alpha}^{1}(0, 1)\right), \quad \forall k \geq 1
\end{array}
$$
and
\begin{equation}\label{bounds Y0 Yk}
\begin{array}{l}
\left\|Y_{0}\right\|_{L^{2}\left(Q_{0} ; \mathbb{R}^{2}\right)} \leq C e^{C T} e^{\frac{M}{T_{1}}}\left(\left\|y_{0}\right\|_{H^{-1}_\alpha}+\|f\|_{L^{2}\left(Q_{0} ; \mathbb{R}^{2}\right)}\right) \\
\left(\text { resp. , }\left\|Y_{0}\right\|_{L^{2}\left(Q_{0}\right) \times C^{0}\left(\overline{Q}_{0}\right)} \leq C e^{C T} e^{\frac{M}{T_{1}}}\left(\left\|y_{0}\right\|_{H^{-1}_\alpha \times H_\alpha^{1}}+\|f\|_{L^{2}\left(Q_{0} ; \mathbb{R}^{2}\right)}\right)\right), \\
\left\|Y_{k}\right\|_{L^{2}\left(Q_{k}\right) \times C^{0}\left(\overline{Q}_{k}\right)} \leq C e^{C T} e^{\frac{M}{T_{k+1}-T_{k}}}\|f\|_{L^{2}\left(Q_{k} ; \mathbb{R}^{2}\right)}, \quad \forall k \geq 1 .
\end{array}
\end{equation}
Let us divide the proof into two cases: the case $k=0$ and the case $k \geq 1$.

\textbf{Case} $k=0$. First, from Theorem \ref{Thm-null-cont1}, we can use that $q T_{1}=T(q-1)$ to obtain (recall that $\left.v_{0}=\mathcal{C}_{T_{1}}^{(0)}\left(y_{0}\right)\right)$
$$
\left\|v_{0}\right\|_{L^{2}\left(0, T_{1}\right)} \leq C_{0} e^{\frac{M}{T_{1}}}\left\|y_{0}\right\|_{H^{-1}_\alpha}=C_{0} e^{\frac{M q(p+1)}{(q-1) T}} \rho_{0}\left(T_{1}\right)\left\|y_{0}\right\|_{H^{-1}_\alpha}
$$
Using now that $\rho_{0}$ is a positive continuous non-increasing function, from the previous estimate, we deduce the existence of a positive constant $C$ such that
\begin{equation}\label{bound v0/rho0}
 \left\|\frac{v_{0}}{\rho_{0}}\right\|_{L^{2}\left(0, T_{1}\right)} \leq C e^{\frac{C}{T}}\left\|y_{0}\right\|_{H^{-1}_\alpha}
\end{equation}
On the other hand, from \eqref{bounds Y0 Yk},
\begin{equation}
\begin{array}{llll}
\left\|Y_{0}\right\|_{L^{2}\left(Q_{0} ; \mathbb{R}^{2}\right)}  \leq  C e^{C T} e^{\frac{M}{T_{1}}}\left(\left\|y_{0}\right\|_{H^{-1}_\alpha}+\|f\|_{L^{2}\left(Q_{0}: \mathbb{R}^{2}\right)}\right) \\\\
\qquad\qquad\qquad=C e^{C T} e^{\frac{M q(p+1)}{(q-1)T}} \rho_{0}\left(T_{1}\right)\left(\left\|y_{0}\right\|_{H^{-1}_\alpha}+\|f\|_{L^{2}\left(Q_{0} ; \mathbb{R}^{2}\right)}\right) \\
\text { (resp., } \left.\left\|Y_{0}\right\|_{L^{2}\left(Q_{0}\right) \times C^{0}\left(\overline{Q}_{0}\right)}  \leq  C e^{C T}  e^{\frac{M q(p+1)}{(q-1) T}} \rho_{0}\left(T_{1}\right)\left(\left\|y_{0}\right\|_{H^{-1}_\alpha \times H_{\alpha}^{1}}+\|f\|_{L^{2}\left(Q_{0}: \mathbb{R}^{2}\right)}\right)\right) .
\end{array}
\end{equation}
Observe that $\|f\|_{L^{2}\left(Q ; \mathbb{R}^{2}\right)} \leq\|f\|_{\mathcal{F}}$ (see the expression of $\rho_{\mathcal{F}}$ in \eqref{weight functions}). Hence, repeating the previous argument, we get
\begin{equation}
\begin{array}{lll}
&\left\|\frac{Y_{0}}{\rho_{0}}\right\|_{L^{2}\left(Q_{0} ; \mathbb{R}^{2}\right)} \leq C e^{C\left(T+\frac{1}{T}\right)}\left(\left\|y_{0}\right\|_{H^{-1}_\alpha}+\|f\|_{\mathcal{F}}\right) \\
&\text { (resp., } \left.\left\|\dfrac{Y_{0}}{\rho_{0}}\right\|_{L^{2}\left(Q_{0}\right) \times C^{0}\left(\overline{Q}_{k}\right)} \leq C e^{C\left(T+\frac{1}{T}\right)}\left(\left\|y_{0}\right\|_{H^{-1}_\alpha \times H_{\alpha}^{1}}+\|f\|_{\mathcal{F}}\right)\right)
\end{array}
\end{equation}
\textbf{Case} $k \geq$ 1. Again, taking into account formula $v_{k}=\mathcal{C}_{T_{k+1}-T_{k}}^{(0)}\left(a_{k}\right)$, Theorem \ref{Thm-null-cont1}, \eqref{bound y tilde} and \eqref{weight functions link}, we infer
$$\begin{aligned}
\left\|v_{k}\right\|_{L^{2}\left(T_{k}, T_{k+1}\right)} \leq C e^{\frac{M}{T_{k+1}-T_{k}}}\left\|a_{k}\right\|_{H^{-1}} &\leq C e^{C T} e^{\frac{M}{T_{k+1}-T_{k}}}\|f\|_{L^{2}}\left(Q_{k-1} ; \mathbb{R}^{2}\right)\\
&=C e^{C T} \frac{\rho_{0}\left(T_{k+1}\right)}{\rho_{\mathcal{F}}\left(T_{k-1}\right)}\|f\|_{L^{2}\left(Q_{k-1} ; \mathbb{R}^{2}\right)} .
\end{aligned}$$
As in the case $k=0$, using the fact that $\rho_{0}$ and $\rho_{\mathcal{F}}$ are non-increasing functions, from the previous inequality, we deduce
\begin{equation}\label{bound vk and rho 0}
\left\|\frac{v_{k}}{\rho_{0}}\right\|_{L^{2}\left(T_{k}, T_{k+1}\right)} \leq C e^{C T}\left\|\frac{f}{\rho_{\mathcal{F}}}\right\|_{L^{2}\left(Q_{k-1} ; \mathbb{R}^{2}\right)}, \quad \forall k \geq 1
\end{equation}
We can also repeat the previous argument to obtain an estimate for $Y_{k}$ when $k \geq 1$. From \eqref{bounds Y0 Yk},
\begin{equation}
\begin{aligned}
\left\|Y_{k}\right\|_{L^{2}\left(Q_{k}\right) \times C^{0}\left(\overline{Q}_{k}\right)}  \leq C e^{C T} e^{\frac{M}{T_{k+1}-T_{k}}}\|f\|_{L^{2}\left(Q_{k} ; \mathbb{R}^{2}\right)} & =C e^{C T} \frac{\rho_{0}\left(T_{k+1}\right)}{\rho_{\mathcal{F}}\left(T_{k-1}\right)}\|f\|_{L^{2}\left(Q_{k} ; \mathbb{R}^{2}\right)} \\
& \leq C e^{C T} \frac{\rho_{0}\left(T_{k+1}\right)}{\rho_{\mathcal{F}}\left(T_{k}\right)}\|f\|_{L^{2}\left(Q_{k} ; \mathbb{R}^{2}\right)}
\end{aligned}
\end{equation}
what implies 
\begin{equation}\label{bounds Yk and rho 0 }
\left\|\frac{Y_{k}}{\rho_{0}}\right\|_{L^{2}\left(Q_{k}\right) \times C^{0}\left(\overline{Q}_{k}\right)} \leq C e^{C T}\left\|\frac{f}{\rho_{\mathcal{F}}}\right\|_{L^{2}\left(Q_{k} ; \mathbb{R}^{2}\right)}, \quad \forall k \geq 1
\end{equation}
With the functions $v_{k}$ and $Y_{k}, k \geq 0$, defined above, we define
$$
\mathcal{C}_{T}^{(1)}\left(y_{0}, f\right):=v=\sum_{k \geq 0} v_{k} 1_{\left[T_{k}, T_{k+1}\right)} \quad \text { and } \quad E_{T}^{(0)}\left(y_{0}, f\right):=Y=\sum_{k \geq 0} Y_{k} 1_{\left[T_{k}, T_{k+1}\right)}
$$
where $1_{I}$ is the characteristic function on the set $I$. Let us first remark that, by construction, $\mathcal{C}_{T}^{(1)}$ and $E_{T}^{(0)}$ are linear operators. On the other hand, recall that $Y_{k}=\tilde{y}_{k}+\hat{y}_{k}, k \geq 0$, where $\tilde{y}_{k}$ and $\hat{y}_{k}$ are respectively the solution to systems \eqref{linear system y tilde} and \eqref{linear system y hat}. So,
\begin{align*}
Y_{k}\left(T_{k+1}^{-}\right)=a_{k+1}=\hat{y}_{k+1}\left(T_{k+1}^{+}\right)=Y_{k+1}\left(T_{k+1}^{+}\right), \quad \forall k \geq 0
\end{align*}
which implies that the function $Y$ is continuous at time $T_{k}$, for any $k \geq 1$, and is the solution of system \eqref{bound-nonhomog-sys} associated to $\left(y_{0}, f, v\right)$.

Finally, thanks to \eqref{bound v0/rho0} and \eqref{bound vk and rho 0}, we also deduce that $\mathcal{C}_{T}^{(1)}\left(y_{0}, f\right) \in \mathcal{V}$ and $E_{T}^{(0)}\left(y_{0}, f\right) \in$ $\mathcal{H}_{0}$ (resp., $\left.E_{T}^{(0)}\left(y_{0}, f\right) \in\mathcal{H}\right)$ for any $\left(y_{0}, f\right) \in H^{-1}_\alpha (0,1)^2 \times \mathcal{F}$ (resp., for any $\left(y_{0}, f\right) \in$ $\left.H^{-1}_\alpha(0, 1) \times H_{\alpha}^{1}(0, 1) \times \mathcal{F}\right)$ and
\begin{align*}
\begin{array}{l}
\left\|\mathcal{C}_{T}^{(1)}\left(y_{0}, f\right)\right\|_{\mathcal{V}}=\|v\|_{\mathcal{V}} \leq C e^{C\left(T+\frac{1}{T}\right)}\left(\left\|y_{0}\right\|_{H^{-1}_\alpha}+\|f\|_{\mathcal{F}}\right), \\
\left\|E_{T}^{(0)}\left(y_{0}, f\right)\right\|_{\mathcal{H}_{0}}=\|Y\|_{\mathcal{H}_{0}} \leq C e^{C\left(T+\frac{1}{T}\right)}\left(\left\|y_{0}\right\|_{H^{-1}_\alpha}+\|f\|_{\mathcal{F}}\right), \quad \forall\left(y_{0}, f\right) \in H^{-1}_\alpha(0,1)^2 \times \mathcal{F},
\end{array}
\end{align*}
(resp.,
\begin{align*}
\begin{gathered}
\left\|E_{T}^{(0)}\left(y_{0}, f\right)\right\|_{\mathcal{H}}=\|Y\|_{\mathcal{H}} \leq C e^{C\left(T+\frac{1}{T}\right)}\left(\left\|y_{0}\right\|_{H^{-1}_\alpha \times H_{\alpha}^{1}}+\|f\|_{\mathcal{F}}\right), \\
\left.\forall\left(y_{0}, f\right) \in H^{-1}_\alpha(0, 1) \times H_{\alpha}^{1}(0, 1) \times \mathcal{F}\right)
\end{gathered}
\end{align*}
The above estimates provide the proof of theorem \ref{thm-operators}. This ends the proof.
\end{proof}

\begin{remark}
Note that theorem \ref{thm-operators} provides a null controllability result for the non-homogeneous system \eqref{bound-nonhomog-sys} when $y_{0} \in H^{-1}_\alpha\left(0, 1 ; \mathbb{R}^{2}\right)$ and $f \in \mathcal{F} .$ Indeed, since $\rho_{0}$ is a continuous function on $[0, T]$ satisfying $\rho_{0}(T)=0$, it is clear that
\begin{align*}
y=E_{T}^{(0)}\left(y_{0}, f\right) \in \mathcal{H}_{0} \cap C^{0}\left([0, T] ; H_{\alpha}^{-1}(0,1)^2\right)
\end{align*}
solves \eqref{bound-nonlinear-sys} and satisfies $y(T, \cdot)=0$ in $H_{\alpha}^{-1}(0,1)^2$.
\end{remark}

\subsection{Proof of Theorem \ref{thm-local-exact-control}}
In this section, we will prove the local exact controllability at time \( T > 0 \) of the nonlinear system \eqref{bound-nonlinear-sys}, as stated in Theorem \ref{thm-local-exact-control}. For this purpose, using techniques similar to those in \cite{Burgos_2020} and \cite{Liu_2012}, we shall employ a fixed-point strategy in the space \( \mathcal{H} \). This approach will demonstrate the existence of a control \( v \in \mathcal{V} \) such that the system \eqref{bound-nonlinear-sys} has a solution \( y \in \mathcal{H} \) associated with \( (v, y_0) \), and the null controllability result will follow from the condition \( y \in \mathcal{H} \) for this system.

Let us fix $\delta>0$. We consider $\overline{B}_{\delta}$ the closed ball in the space $\mathcal{F}$ (introduced in \eqref{weighted normed spaces}) by
$$
\overline{B}_{\delta}=\left\{f \in \mathcal{F}:\|f\|_{\mathcal{F}} \leq \delta\right\}.
$$
Assume that  the initial datum ${y}_{0} \in H^{-1}_\alpha(0, 1) \times H_{\alpha}^{1}(0, 1) $ satisfies \eqref{estim-1.15}.

For each $f \in \overline{B}_{\delta} \subset \mathcal{F}$, we denote 
\begin{align*}
v_{f}=\mathcal{C}_{T}^{(1)}\left(y_{0}, f\right) \in \mathcal{V} ~\text{ and }~ y_{f}=\left(y_{1,f},y_{2,f}\right):=E_{T}^{(1)}\left(y_{0}, f\right) \in \mathcal{H}
\end{align*}
where the operators $\mathcal{C}_{T}^{(1)}$ and $E_{T}^{(1)}$ are given in Theorem \ref{thm-operators}.
Based on the above result and \eqref{estim-1.15}, one has
\begin{equation}\label{estimat-5.14}
\left\|y_{f}\right\|_{\mathcal{H}}+\left\|v_{f}\right\|_{\mathcal{V}} \leq C e^{C\left(T+\frac{1}{T}\right)}\left(\left\|y_{0}\right\|_{H^{-1}_\alpha \times H_{\alpha}^{1}}+\|f\|_{\mathcal{F}}\right) \leq C e^{C\left(T+\frac{1}{T}\right)} \delta, \quad \forall f \in \overline{B}_{\delta}
\end{equation}
for a positive constant $C$. Thus, we define the nonlinear operator 
\begin{align*}
\mathcal{O}: \overline{B}_{\delta} \rightarrow C^{0}\left(\overline{Q}_{T} ; \mathbb{R}^{2}\right)
\end{align*}
given by 
\begin{align*}
\mathcal{O} (f)=f(y)=\begin{pmatrix}
f_1(y) \\
f_2(y) 
\end{pmatrix}
\end{align*}
Now, we focus on the key part of the proof, which is to verify the validity of the assumptions of the fixed-point theorem, in order to prove that \( \mathcal{O} \) admits a fixed point \( f \) in the complete metric space \( \overline{B}_{\delta} \subset \mathcal{F} \). First, let \( p, q > 1 \) be fixed such that
\begin{align*}
q^{2} \in\left(1, \frac{2p}{p+1}\right)
\end{align*}
With this choice, the function \( \frac{\rho_0^2}{\rho_{\mathcal{F}}} \) is uniformly bounded on \( [0, T] \), i.e., there exists a constant \( C_T > 0 \), depending on \( T \), such that
\begin{align*}
\left\|\frac{\rho_{0}^2}{\rho_{\mathcal{F}}}\right\|_{C^{0}[0, T]} \leq C_{T}
\end{align*}
Let us now verify the assumptions of the Banach Fixed-Point Theorem, taking into account that the source term \( f \) satisfies \eqref{nonlinCondit1}:
\begin{enumerate}
\item $\mathcal{O}\left(\overline{B}_{\delta}\right) \subset \overline{B}_{\delta}$ : Indeed, if $f \in \overline{B}_{\delta}$, then, from \eqref{estimat-5.14}, we obtain
\begin{align*}
\begin{aligned}
\|\mathcal{O}(f)\|_{\mathcal{F}} & \leq C_{T}\left\|\frac{\mathcal{O}(f)}{\rho_{\mathcal{F}}}\right\|_{C^{0}\left(\overline{Q}_{T} ; \mathbb{R}^{2}\right)} \leq C_{T}\left(\left\|\frac{f_1}{\rho_{\mathcal{F}}}\right\|_{C^{0}\left(\overline{Q}_{T}\right)}+\left\|\frac{f_2}{\rho_{\mathcal{F}}}\right\|_{C^{0}\left(\overline{Q}_{T}\right)}\right) \\
&\leq C_{T}\left(\left\|\frac{y_{2}^2}{\rho_{\mathcal{F}}}\right\|_{C^{0}\left(\overline{Q}_{T}\right)}+\left\|\frac{y_{2}^2}{\rho_{\mathcal{F}}}\right\|_{C^{0}\left(\overline{Q}_{T}\right)}\right) \\
& \leq C_{T}\left(\left\|\frac{\rho_{0}^{2}}{\rho_{\mathcal{F}}}\right\|_{C^{0}\left(\overline{Q}_{T}\right)}\left\|\frac{y_{2}}{\rho_{0}}\right\|_{C^{0}\left(\overline{Q}_{T}\right)}^{2}\right) \\
& \leq C_{T} \left\|y_{f}\right\|_{\mathcal{H}}^{2}\leq C_{T} e^{C\left(T+\frac{1}{T}\right)}\delta^{2} \leq \delta
\end{aligned}
\end{align*}
for $\delta=\delta(T)$ small enough.
\item $\mathcal{O}$ is a contraction mapping: Let us take $f_{1}, f_{2} \in \overline{B}_{\delta} \subset \mathcal{F}$ and denote $$y^{(1)}=\left( y_1, y_2\right)=E_{T}^{(1)}\left(y_{0}, f_{1}\right) \in \mathcal{H}$$
$$y^{(2)}=\left(y_1', y_2'\right)=E_{T}^{(1)}\left(y_{0}, f_{2}\right) \in \mathcal{H}$$
Thus, using again \eqref{estimat-5.14} and Theorem \ref{thm-operators}, we have
\begin{align*}
\begin{aligned}
\left\|\mathcal{O}\left(f_{1}\right)-\mathcal{O}\left(f_{2}\right)\right\|_{\mathcal{F}} &\leq
 C_{T} \sum_{j=1}^{2}\left\|\frac{f_{j}^{(1)}\left(y_{1},y_{2}\right)-f_{j}^{(2)}\left(y_{1}',y_{2}'\right)}{\rho_{\mathcal{F}}}\right\|_{C^{0}\left(\overline{Q}_{T};\mathbb{R}^2\right)} \\
&\leq C_{T}\left\| \frac{y_{2}^2-{y_2'}^2}{\rho_{\mathcal{F}}}\right\|_{C^{0}\left(\overline{Q}_{T}\right)}\\
&\leq C_{T}\left\|\frac{\rho_{0}}{\rho_{\mathcal{F}}}(y_2+y_{2}') \frac{y_{2}-y_{2}'}{\rho_{0}}\right\|_{C^{0}\left(\overline{Q}_{T}\right)} \\
&\leq C_{T}\left\|\frac{\rho_{0}^2}{\rho_{\mathcal{F}}} \times \frac{y_2+y_{2}'}{\rho_{0}} \right\|_{C^{0}\left(\overline{Q}_{T}\right)}  \left\| \frac{y_{2}-y_{2}'}{\rho_{0}}\right\|_{C^{0}\left(\overline{Q}_{T}\right)} \\
&\leq C_{T} \left( \left\| y^{(1)} \right\|_{\mathcal{H}} +\left\|y^{(2)}\right\|_{\mathcal{H}} \right) \left\|E_{T}^{(1)}\left(y_{0}, f_{1}\right)-E_{T}^{(1)}\left(y_{0}, f_{2}\right)\right\|_{\mathcal{H}} \\
& \leq C_{T} e^{C\left(T+\frac{1}{T}\right)} \delta \left\|f_{1}-f_{2}\right\|_{\mathcal{F}}
\end{aligned}
\end{align*}
From this inequality, it is clear that we can choose \( \delta = \delta(T) \) small enough such that \( \mathcal{O} \) becomes a contraction mapping. This proves that the operator \( \mathcal{O} \) has a fixed point. 
On the other hand, using the fact that \( \mathcal{O} \) admits a fixed point \( f \in \mathcal{F} \), we have \( y_f \in \mathcal{H} \). Together with \( v_f \in \mathcal{V} \), this provides a solution to the system \eqref{bound-nonlinear-sys} associated with the initial data \( y_0 = (y_1^0, y_2^0) \). In fact, from Proposition \ref{well.posed.transp}, we obtain \( y_f \in C^0([0, T]; H_\alpha^{-1}(0,1)^2) \). Finally, the condition \( y_f \in C^0([0, T]; H_\alpha^{-1}(0,1)^2) \cap \mathcal{H} \) implies the null controllability result for system \eqref{bound-nonlinear-sys}. This completes the proof of Theorem \ref{thm-local-exact-control}.

\end{enumerate}

%

\appendix
\section{Appendix}\label{appendix:well posed}
For the readers’ convenience, in this section we prove some results used throughout the
previous sections.

\subsection{Proof of Proposition \ref{well.posed.transp}}
This appendix will be devoted to dealing with the existence and uniqueness of solution for the non-homogeneous linear system \eqref{bound-nonhomog-sys} with term source $f \in L^{2}\left(0, 1 \right)^{2}$. To be precise, we will prove Proposition \ref{well.posed.transp}.
\begin{proof}
First of all, we denote by $\mathcal{H}_\alpha^1(0,1)$ the following weighted Sobolev space:
\begin{align*}
\mathcal{H}_\alpha^1(0,1):=\left\{y \in L^2(0,1) \cap H_{l o c}^1((0,1]): x^{\alpha / 2} y_x \in L^2(0,1)\right\} .
\end{align*}
Let $y_0 \in H_\alpha^{-1}(0,1)^2, v \in L^2(0, T)$ and consider the following functional
\begin{align*}
\mathcal{T}: L^2(Q_T)^2 \rightarrow \mathbb{R}
\end{align*}
given by
\begin{align*}
\mathcal{T}(g)=\left\langle y_0, \varphi(0, \cdot)\right\rangle_{H_\alpha^{-1}, H_\alpha^1}-\int_0^T B^*\left(x^\alpha \varphi_x\right)(t, 1) v(t) d t+ \iint_{Q_T} f.\varphi \, dx dt
\end{align*}
where $\varphi \in C^0\left([0, T] ; H_\alpha^1(0,1)^2\right) \cap L^2\left(0, T ; H_\alpha^2(0,1)^2\right)$ is the solution of the adjoint system \eqref{adjoint system} associated to $g \in L^2(Q_T)^2$. From the estimate \eqref{bound phi prop 2.1}, it follows that
\begin{align*}
|\mathcal{T}(g)| \leq C\left(\|v\|_{L^2(0, T)}+\left\|y_0\right\|_{H_\alpha^{-1}} +\left\|f\right\|_{L^2(L^2)}
\right)\|g\|_{L^2(Q)^2}
\end{align*}
for all $g \in L^2(Q)^2$. Hence, $\mathcal{T}$ is bounded. As a consequence, by Riesz Representation Theorem, there exists a unique $y \in L^2(Q)^2$ satisfying \eqref{duality identity}. Moreover,
\begin{align*}
\|y\|_{L^2(Q)^2}=\|\mathcal{T}\| \leq C\left(\|v\|_{L^2(0, T)}+\left\|y_0\right\|_{H_\alpha^{-1}} +\left\|f\right\|_{L^2(L^2)}\right)
\end{align*}
and $y$ satisfies the equality $y_t-\left(x^\alpha y_x\right)_x - A y = f \quad$ in $\quad \mathcal{D}^{\prime}(Q_T)^2$.
Next, we are going to prove that the solution $y$ of system \eqref{bound-sys} is more regular. To be precise, we show that $\left(x^\alpha y_x\right)_x \in L^2\left(0, T ;\left(H_\alpha^2(0,1)^2\right)^{\prime}\right)$ and
\begin{align}\label{stima appendix}
\left\|\left(x^\alpha y_x\right)_x\right\|_{L^2\left(\left(H_\alpha^2(0,1)^2\right)^{\prime}\right)} \leq C\left(\|v\|_{L^2(0, T)}+\left\|y_0\right\|_{H_\alpha^{-1}}  +\left\|f\right\|_{L^2(L^2)}\right)
\end{align}
For doing that, let us take two sequences $\left\{y_0^m\right\}_{m \geq 1} \subset H_\alpha^1(0,1)^2$ and $\left\{v^m\right\}_{m \geq 1} \subset H_0^1(0, T)$ such that
\begin{align*}
y_0^m \rightarrow y_0 \quad \text { in } H_\alpha^{-1}(0,1)^2 \quad \text { and } \quad v^m \rightarrow v \quad \text { in } \quad L^2(0, T)
\end{align*}
Now, the strategy consists in transforming system \eqref{bound-sys} into a problem with homogeneous boundary condition and a source term. To this end, let us introduce the change of variables
\begin{align*}
y^m(t, x)=\tilde{y}^m(t, x)+x^{2-\alpha} v^m(t) B
\end{align*}
where $y^m$ is the solution of \eqref{bound-sys} associated to $y_0^m$ and $v^m$. Then, formally, the new function $\tilde{y}^m$ satisfies the problem
\begin{equation}\label{mbound-sys}
\left\{
\begin{array}{lll}
\tilde{y}_{t}^m - ( x^{\alpha}  \tilde{y}^m_x )_x - A \tilde{y}^m = \tilde{f}^m ,  &  & \text{in} \;  Q_T:=(0, \, T)\times(0, 1), \\
\tilde{y}^m(t,1)= 0, & & \text{in} \; (0, \, T),  \\
\begin{cases}
& \tilde{y}^m(t, 0) = 0,   \qquad   \, 0 \leq \alpha < 1 \\
&  x^{\alpha} \tilde{y}^m_x  (t, 0)= 0 , \quad 1\leq \alpha < 2 \\
\end{cases}  & & t \in (0, T),\\
\tilde{y}^m(0,x)= \tilde{y}^m_{0}(x),  & & \text{in}  \;  (0, 1),
\end{array}
\right.
\end{equation}
where $\tilde{f}^m(t, x)=\left[(2-\alpha) v^m(t)-x^{2-\alpha} v_t^m(t)\right] B+x^{2-\alpha} v^m(t) A B$. Moreover, an easy computation shows that the function $x \mapsto x^{2-\alpha}$ belongs to $\mathcal{H}_\alpha^1(0,1)$ which of course implies that $\tilde{f}^m \in L^2(Q_T)^2$. In view of the previous regularity assumptions we can apply Proposition \ref{prop.2.1}, to deduce that system \eqref{mbound-sys} has a unique solution
\begin{align*}
\tilde{y}^m \in L^2\left(0, T ; H_\alpha^2(0,1)^2\right) \cap C^0\left(0, T ; H_\alpha^1(0,1)^2\right)
\end{align*}
Therefore, the problem \eqref{bound-sys} for $v^m$ and $y_0^m$ has a unique solution $y^m \in L^2\left(0, T ; \mathcal{H}_\alpha^1(0,1)^2\right)$ which satisfies
\begin{align*}
\iint_Q y^m \cdot g d t d x=\left\langle y_0^m, \varphi(0, \cdot)\right\rangle_{H_\alpha^{-1}, H_\alpha^1}-\int_0^T B^*\left(x^\alpha \varphi_x\right)(t, 1) v^m(t) d t +\iint_{Q_T} f.\varphi\, dx dt, \quad \forall m \geq 1
\end{align*}
for all $g \in L^2(Q_T)^2$, where $\varphi$ is the solution of the system \eqref{adjoint system} associated to $g$. Using this last identity and \eqref{duality identity}, we obtain
\begin{align}\label{property of ym}
\left\{\begin{array}{l}
\left\|y^m\right\|_{L^2(Q)^2} \leq C\left(\|v\|_{L^2(0, T)}+\left\|y_0\right\|_{H_\alpha^{-1}}\right) \\
y^m \rightarrow y \quad \text { in } L^2(Q)^2 \quad \text { and } \quad\left(x^\alpha y_x^m\right)_x \rightarrow\left(x^\alpha y_x\right)_x \quad \text { in } \mathcal{D}^{\prime}(Q_T)^2
\end{array}\right.
\end{align}
On the other hand, integrations by parts lead to
\begin{align*}
\int_0^T\left\langle\left(x^\alpha y_x^m\right)_x, \psi\right\rangle d t=\iint_Q y^m \cdot\left(x^\alpha \psi_x\right)_x d t d x-\int_0^T B^*\left(x^\alpha \psi_x\right)(t, 1) v^m(t) d t, \quad \forall m \geq 1
\end{align*}
for every $\psi \in L^2\left(0, T ; H_\alpha^2(0,1)^2\right)$. From this equality we infer that the sequence $\left(x^\alpha y_x^m\right)_x$ is bounded in $L^2\left(0, T ;\left(H_\alpha^2(0,1)^2\right)^{\prime}\right)$. Combining with \eqref{property of ym}, we deduce that $\left(x^\alpha y_x\right)_x$ belongs to $L^2\left(0, T ;\left(H_\alpha^2(0,1)^2\right)^{\prime}\right)$ and satisfies estimate \eqref{stima appendix}. With the previous property in mind and the identity $y_t=\left(x^\alpha y_x\right)_x+A y+f$, we also see that $y_t \in L^2\left(0, T ;\left(H_\alpha^2(0,1)^2\right)^{\prime}\right)$ and
\begin{align*}
\left\|y_t\right\|_{L^2\left(\left(H_\alpha^2(0,1)^2\right)^{\prime}\right)} \leq C\left(\|v\|_{L^2(0, T)}+\left\|y_0\right\|_{H_\alpha^{-1}} +\left\|f\right\|_{L^2(L^2)} \right)
\end{align*}
for a positive constant $C$. Therefore $y \in C\left([0, T] ; X^2\right)$, where $X$ is the interpolation space $X=\left[L^2(0,1),\left(H_\alpha^2(0,1)\right)^{\prime}\right]_{1 / 2}=H_\alpha^{-1}(0,1)$ (see  \cite[Theorem 11.4]{chipot}). In conclusion, we get
\begin{align*}
\|y\|_{C\left(H_\alpha^{-1}(0,1)^2\right)} \leq C\left(\|v\|_{L^2(0, T)}+\left\|y_0\right\|_{H_\alpha^{-1}} +\left\|f\right\|_{L^2(L^2)} \right)
\end{align*}
Finally, one can easily check that $y(0, \cdot)=y_0 \quad$ in $H_\alpha^{-1}(0,1)^2$. This ends the proof.
\end{proof}

\subsection{Spectral properties}\label{appendix:a}
The spectral properties of the operators $L$ (given by \eqref{operatorL}) and its adjoint $L^*$ will be addressed in this appendix. To be more specific, we shall demonstrate Propositions \ref{propertieseigenfLL} and \ref{rieszbasisresult} presented in section \ref{Section-eigenvalue}.
\begin{proof}[Proof of Proposition \ref{propertieseigenfLL}]
The proof of the proposition above is similar to the one given, for instance in \cite[Proposition 3.4]{AHSS2021}.
We shall demonstrate the result for the operator $L$. The argument for its adjoint $L^*$ is provided by the same logic. 

We look for a complex $\lambda$ and a function $\psi \in H^2(0, 1; \mathbb{C}^2)\cap H_\alpha^1(0, 1;\mathbb{C}^2)$ such that $$L(\psi) = \lambda \psi.$$ Using the fact that the function $\Phi_{\nu_{\alpha}, n}$ is the eigenfunction of the degenerate elliptic operator $- \partial_x (x^{\alpha} \partial_x \cdot )$ associated to the eigenvalues $\lambda_{\nu_{\alpha}, n}= \kappa_{\alpha}^2 j_{\nu_{\alpha}, n}^2$,  we can find $\psi$ as
$$ \psi(x)= \sum_{n\geq 1} a_n \Phi_{\nu_{\alpha}, n}(x), \quad \forall x \in (0,1),$$
where $\{a_n\}_{n\geq1} \subset \mathbb{C}^2$ and, for some $k\geq 1$, $a_k \neq 0$. From the identity $L(\psi) = \lambda \psi$ we deduce 
$$  \sum_{n\geq 1} \big(\kappa_{\alpha}^2 j_{\nu_{\alpha}, n}^2 I - A - \lambda I \big) a_n \Phi_{\nu_{\alpha}, n}(x)=0, \quad \forall x \in (0,1).$$

From this identity, it is clear that the eigenvalues of the operator $L$ correspond to the eigenvalues of the matrices
\begin{equation*}
\kappa_{\alpha}^2 j_{\nu_{\alpha}, n}^2 I - A, \qquad \forall n \geq 1
\end{equation*}
and the associated  eigenfunctions of $L$ are given
choosing $a_n = z_k \delta_{kn}$, for any $n \geq1$, where $z_k \in \mathbb{C}^2$ is an associated eigenvector of $j_{\nu_{\alpha}, n}^2 I - A$, that is to say, $\psi_n(\cdot) = z_n \Phi_{\nu_{\alpha}, n}(\cdot)$.

Taking into account the expression of the characteristic polynomial of $\kappa_{\alpha}^2 j_{\nu_{\alpha}, n}^2 I - A$:
$$P(z) = z^2 - z (2 \lambda_{\nu_{\alpha}, n} - a_2) + \lambda_{\nu_{\alpha}, n} (\lambda_{\nu_{\alpha}, n} - a_2) - a_1, \qquad n \geq 1,$$
a direct computation provides the formulas \eqref{sigmaLL*} and \eqref{eigenf-L} as eigenvalues and associated eigenfunctions
of the operator L. This ends the proof.
\end{proof}
 
\begin{proof}[Proof of Proposition \ref{rieszbasisresult}] 
From the expressions of $\psi_n^{(i)}$ and $\Psi_n^{(i)}$, we can write
\begin{align*}
\psi_n^{(i)} = U_{i} \Phi_{\nu_{\alpha}, n} \quad \text{and}\quad \Psi_n^{(i)} = V_{i} \Phi_{\nu_{\alpha}, n},\quad i=1,2,\quad  n\geq 1,
\end{align*}
where $ U_{i}, V_{i} \in \mathbb{R}^2$ and $\Phi_{\nu_{\alpha}, n}$ is given \eqref{eigenfunctions}.
\smallskip
\begin{enumerate}
\item It is not difficult to check that $\{U_i\}_{i=1,2}$ and $\{V_i\}_{i=1,2}$ are
biorthogonal families of $\mathbb{R}^2$. Moreover,
since $(\Phi_{\nu_{\alpha}, n})_{n\geq1}$ is an orthonormal basis for $L^2(0,1)$, we readily deduce
\begin{align*}
\langle \psi_n^{(i)},\Psi_k^{(j)} \rangle =  (U_{i})^{tr} V_{j} \langle \Phi_{\nu_{\alpha}, n},\Phi_{\nu_{\alpha}, k} \rangle = \delta_{ij} \delta_{nk}, \quad \forall n,k\geq 1,\quad i,j=1, 2.
\end{align*}
This proves the claim.
\item We will use \cite[Lemma 1.44]{chris}. For this purpose, let us consider $f= (f_1, f_2)^{tr} \in L^2(0,1)^2$ such that
\begin{align*}
\langle f  , \psi_n^{(i)} \rangle = 0,\quad \forall n\geq 1, \quad i=1, 2.
\end{align*}
If
we denote $f_{i,n}$ $(i = 1,2)$ the corresponding Fourier coefficients of the function $f_i\in L^2(0,1)$ with respect to the basis $(\Phi_{\nu_{\alpha}, n})_{n\geq1}$, then the previous equality can be written as
\begin{align*}
(f_{1,n} , f_{2,n} ) [U_{1} | U_{2}]= 0_{\mathbb{R}^{2}}, \quad \forall n\geq 1.
\end{align*}
Using the fact that $det [U_{1} | U_{2}] \neq 0$, we deduce $f_{1,n} = f_{2,n} = 0$, for all $n\geq 1$. This implies that $f_1=f_2=0$ (since $(\Phi_{\nu_{\alpha}, n})_{n\geq1}$ is an orthonormal basis in $L^2(0,1)$) and, therefore,
$f = 0$ which proves the completeness of $\mathcal{B}$.
A similar argument can be used for $\mathcal{B}^*$ and the conclusion follows immediately.
\item By \cite[Theorem 7.13]{chris}, we know that $\big\{\psi_n^{(1)} , \psi_n^{(2)}\big\}_{n \geq 1}$ is a Riesz basis for $L^2(0,1)^2$ if and only if
$\big\{\psi_n^{(1)} , \psi_n^{(2)}\big\}_{n \geq 1}$ is a complete Bessel sequence and possesses a biorthogonal system
that is also a complete Bessel sequence. Using the previous properties $1)$ and $2)$, we only have to prove that the sequence $\big\{\psi_n^{(1)} , \psi_n^{(2)}\big\}_{n \geq 1}$ and $\big\{\Psi_n^{(1)} , \Psi_n^{(2)}\big\}_{n \geq 1}$ are Bessel sequences. This amounts to prove that the series
\begin{align*}
S_1(f)= \sum_{n\geq 1} \big[ \langle f  , \psi_n^{(1)} \rangle^2 + \langle f  , \psi_n^{(2)} \rangle^2 \big] \quad
\text{and}
\quad S_2(f)= \sum_{n\geq 1} \big[ \langle f  , \Psi_n^{(1)} \rangle^2 + \langle f  , \Psi_n^{(2)} \rangle^2 \big]
\end{align*}
converge for any $f= (f_1, f_2)^{tr} \in L^2(0,1)^2$.

From the definition of the functions $\psi_n^{(i)}$ and $\Psi_n^{(i)}$, it is easy to see that there exists some constant $C>0$ such that
\begin{align*}
S_1(f) \leq C \sum_{n\geq 1} \big( |f_{1,n}|^2 + |f_{2,n}|^2 \big) \quad
\text{and} \quad
S_2(f) \leq C \sum_{n\geq 1} \big( |f_{1,n}|^2 + |f_{2,n}|^2 \big).
\end{align*}
Recall that $f_{i,n}$ is the Fourier coefficient of the function $f_{i}\in L^2(0,1)$ ($i=1,2$) with respect to $\Phi_{\nu_{\alpha}, n}$. Accordingly, the series $S_1(f)$ and $S_2(f)$ converge since $(\Phi_{\nu_{\alpha}, n})_{n\geq1}$ is an orthonormal basis for $L^2(0,1)$. We obtain thus the proof of desired result.
\item For showing item $4)$ we make use of \cite[Theorem 5.12]{chris}.
First, we have that
\begin{align*}
H_{\alpha}^1(0,1) \subset  L^2(0,1)  \subset \big( H_{\alpha}^1(0,1) \big)^{'} = H_{\alpha}^{-1}(0,1).
\end{align*}
Therefore, $\mathcal{B}^* \subset H_{\alpha}^1(0,1)^2 $ and is complete in this space since it is in $L^2(0,1)^2$.
On the other hand, by the definition of the duality pairing, we have
\begin{align*}
\langle \psi_n^{(i)},\Psi_k^{(j)} \rangle_{H_{\alpha}^{-1},H_{\alpha}^1} = \langle \psi_n^{(i)},\Psi_k^{(j)} \rangle=\delta_{ij} \delta_{nk},\quad \forall n,k\geq 1, \quad i,j=1,2.
\end{align*}
Thus, $\mathcal{B} \subset H_{\alpha}^{-1}(0,1)^2 $ and is biorthogonal to $\mathcal{B}^*$, which also yields that $\mathcal{B}^*$ is minimal in $H_{\alpha}^{1}(0,1)^2$ thanks to \cite[Lemma 5.4]{chris}.
To conclude the proof, it remains to prove that for any $f=(f_1,f_2)\in H_{\alpha}^1(0,1)^2$, the series
\begin{align*}
S(f)= \sum_{n\geq 1} \big[ \langle \psi_n^{(1)}, f \rangle_{H_{\alpha}^{-1},H_{\alpha}^1} \Psi_n^{(1)} + \langle  \psi_n^{(2)}, f \rangle_{H_{\alpha}^{-1},H_{\alpha}^1} \Psi_n^{(2)}\big]
\end{align*}
converges in $H_{\alpha}^1(0,1)^2 $.

Using again the definitions of $ \psi_n^{(i)}$ and $\Psi_n^{(i)}$, one can prove that
\begin{equation}
\langle \psi_n^{(1)}, f \rangle_{H^{-1}_\alpha, H_{\alpha}^{1}} \Psi_n^{(1)} =  \frac{1}{\alpha_1 - \alpha_2} \begin{pmatrix} -\alpha_2 f_{1,n}  + f_{2,n}  \\\\ - \alpha_1 \alpha_2 f_{1,n} + \alpha_1 f_{2,n} \end{pmatrix}   \Phi_{\nu_\alpha, n}
\end{equation}
and
\begin{equation}
\langle  \psi_n^{(2)}, f \rangle_{H^{-1}_\alpha, H_{\alpha}^{1}} \Psi_n^{(2)} = \frac{1}{\alpha_2 - \alpha_1} \begin{pmatrix} -\alpha_1 f_{1,n}  + f_{2,n}  \\\\ - \alpha_1 \alpha_2 f_{1,n} + \alpha_2 f_{2,n} \end{pmatrix}  \Phi_{\nu_\alpha, n}
\end{equation}
where $f_{i,n}$ is the Fourier coefficient of the function $f_i \in H_{\alpha}^1(0,1)$, $i= 1,2$.

But, we know that the series $ \sum\limits_{n\geq 1} f_{i,n} \Phi_{\nu_{\alpha}, n}$, $i= 1,2$ converges in $H_{\alpha}^1(0,1)$ since $(\Phi_{\nu_{\alpha}, n})_{n\geq1} $ is an orthogonal basis for $H_{\alpha}^1(0,1)$ and $f_1, f_2 \in H_{\alpha}^1(0,1)$.
This implies that, the series $\sum\limits_{n\geq 1} \langle \psi_n^{(1)}, f \rangle_{H_{\alpha}^{-1},H_{\alpha}^1} \Psi_n^{(1)}$ and $\sum\limits_{n\geq 1} \langle  \psi_n^{(2)}, f \rangle_{H_{\alpha}^{-1},H_{\alpha}^1} \Psi_n^{(2)}$ converge in $H_{\alpha}^1(0,1)^2$ and assure the convergence of $S(f)$ in $H_{\alpha}^1(0,1)^2$. This concludes the proof of the result.
\end{enumerate}
\end{proof}

\noindent{\Large{\bf Acknowledgments}}

\

\noindent This paper has been partially supported by the State Research Agency of the Spanish Ministry of Science and FEDER-EU, project PID2022-137228OB-I00 (MICIU / AEI  / 10.13039/501100011033).

\end{document}